\documentclass[12pt]{article}
\usepackage[a4paper,left=2cm,right=2cm,top=2cm,bottom=2cm]{geometry}
\usepackage[utf8]{inputenc}
\usepackage{amsthm}
\usepackage{amsmath,mathtools}
\usepackage{amssymb}
\usepackage{changepage}
\usepackage[hyperfootnotes=false]{hyperref}
\usepackage{soul,xcolor} 
\usepackage[nottoc,notlof,notlot]{tocbibind}

\usepackage{enumitem}
\makeatletter
\def\namedlabel#1#2{\begingroup
    #2%
    \def\@currentlabel{#2}%
    \phantomsection\label{#1}\endgroup
}
\makeatother

\usepackage{tikz}
\usetikzlibrary{decorations.markings,graphs,graphs.standard,arrows,quotes,topaths,calc}
\usepackage{pgfplots}
\pgfplotsset{compat=1.7}
\usetikzlibrary{fadings}

\tikzgraphsset{declare={polygon_n}{[clique]\foreach\x in\tikzgraphV{\x/}}}

\usepackage{graphicx}
\usepackage{caption}
\usepackage{subcaption}
\usepackage{wrapfig}

\usepackage{hyperref}

\usepackage{authblk}

\newtheoremstyle{theorem}
{10pt}
{3pt}
{\itshape}
{}
{\bfseries}
{.}
{.5em}
{}

\newtheorem*{rep@theorem}{\rep@title}
\newcommand{\newreptheorem}[2]{%
\newenvironment{rep#1}[1]{%
 \def\rep@title{#2 \ref{##1}}%
 \begin{rep@theorem}}%
 {\end{rep@theorem}}}
\makeatother

\theoremstyle{theorem}
\newtheorem{defn}{Definition}[section]
\newtheorem{lem}[defn]{Lemma}
\newtheorem{thm}[defn]{Theorem}
\newtheorem{claim}[defn]{Claim}

\newtheorem{cor}[defn]{Corollary}
\newtheorem{question}[defn]{Question}
\newtheorem{prop}[defn]{Proposition}
\newreptheorem{thm}{Theorem}
\newreptheorem{lem}{Lemma}
\newreptheorem{prop}{Proposition}
\theoremstyle{definition}

\title{Separating Path Systems for the Complete Graph}
\author{Belinda Wickes}
\affil{School of Mathematical Science, Queen Mary University of London, \\ London, E1 4NS, UK. \\Contact: b.wickes@qmul.ac.uk}
\date{November 2023}

\begin{document}

\maketitle

\begin{abstract}
For any graph $G$, a \textbf{separating path system} of $G$ is a family of paths in $G$ with the property that for any pair of edges in $E(G)$ there is at least one path in the family that contains one edge but not the other. We investigate the size of the smallest separating path system for $K_n$, denoted $f(K_n)$.

Our first main result is a construction that shows $f(K_n) \leq \left(\frac{21}{16}+o(1)\right)n$ for sufficiently large $n$. We also show that $f(K_n) \leq n$ whenever $n=p,p+1$ for prime $p$. It is known by simple argument that $f(K_n) \geq n-1$ for all $n \in \mathbb{N}$.

A key idea in our construction is to reduce the problem to finding a single path with some particular properties we call a Generator Path. These are defined in such a way that the $n$ cyclic rotations of a generator path provide a separating path system for $K_n$. Hence existence of a generator path for some $K_n$ gives $f(K_n) \leq n$. We construct such paths for all $K_n$ with $n \leq 20$, and show that generator paths exist whenever $n$ is prime.
\end{abstract}

\section{Introduction}
The study of separation problems was initiated by R\'{e}nyi in the 1960s \cite{Renyi}. The problem is to find a minimal family $\mathcal{F}$ of subsets of ground set $[n] = \{1,2,\dotsc,n\}$, so that for every ordered pair of distinct $x,y \in [n]$ there is some $F \in \mathcal{F}$ with either $x \in F$ and $y \notin F$, or $y \in F$ and $x \notin F$.

It is trivial that in this case $|\mathcal{F}| = \lceil \log_2(n) \rceil$. However, by enforcing some structure on our ground set or various restrictions to the members of $\mathcal{F}$, the question opens up to be an interesting problem. One particularly interesting way of doing this is to have the ground set be vertices or edges of a graph, and the separators inherit certain properties from the graph (see \cite{Cai},\cite{Bollobas Scott},\cite{Foucaud Kovse}). In this paper we focus on what separation means in the context of a ground set of graph edges, and when our separators are restricted to being paths. Problems with such a focus were introduced by Balogh, Csaba, Martin, and Pluh\'{a}r \cite{B et al} as well as Falgas-Ravry, Kittipassorn, Kor\'{a}ndi, Letzter, and Narayanan \cite{F-R et al}. We begin by outlining exactly what we mean to separate a graph by paths.

\begin{defn}\label{def sep}
Let $G$ be a graph and $e,e' \in E(G)$, and let $P \subseteq E(G)$. We say that $P$ \textbf{separates} $e$ and $e'$ if we have $e \in P$ and $e'\notin P$, or $e' \in P$ and $e\notin P$.
Let $\mathcal{S}$ be a family of subsets of $E(G)$ such that for any distinct edges $e,e' \in E(G)$ there is some $P \in \mathcal{S}$ which separates  $e$ and $e'$, then we say that $\mathcal{S}$ is a \textbf{separating system} for $G$.
If we also have the condition that every element of $\mathcal{S}$ is a path in $G$, then we call $\mathcal{S}$ a \textbf{separating path system} of $G$.
\end{defn}

We will use the notation $f(G)$ to mean the size of the smallest path separating system for a graph $G$ as per \cite{F-R et al}. Straight away we notice that $E(G)$ is itself a separating path system where all the paths consist of a single edge. So in particular we get $f(G) \leq |E(G)|$. We also get the bound $f(G) \geq \log_2(|E(G)|)$ from the initial separating property without considering paths.

In \cite{F-R et al} the authors find bounds on $f(G)$ for a selection of graphs $G$ including trees and certain random graphs. They also ask about the case where $G$ is the complete graph on $n$ vertices, that is to determine the exact value of $f(K_n)$. The current best known bounds are
\begin{equation*}
    n-1 \leq f(K_n) \leq 2n+4.
\end{equation*}
The lower bound comes from a simple counting argument found in \cite{F-R et al}, the upper bound is a direct consequence of a result in \cite{B et al}.

We note that the version of the problem we are considering (and which was the subject of \cite{F-R et al}) is sometimes called `weak separation'. The case where we have a family $\mathcal{F}$ with $F,F' \in \mathcal{F}$ where $x \in F$, $y \notin F$ and $y \in F'$, $x \notin F'$ for all unordered distinct $x,y \in [n]$, is called a `strongly separating' system. Clearly, if a family of paths strongly separates every pair, then it also weakly separates all pairs, and hence upper bounds for the strong variation also upper bound the weak.

In \cite{B et al} the authors use a probabilistic argument to give an upper bound of $f'(K_n)\leq 2n+4$, where $f'(G)$ is the strong analogue of $f(G)$. Because any strongly separating system is also inherently a weak separating system, and  therefore $f(G) \leq f'(G)$ for any graph $G$, this means we have $f(K_n)\leq 2n+4$.  We know of no explicit work towards determining an upper bound for $f(K_n)$ in the weak setting.

The main idea used in the probabilistic upper bound in \cite{B et al} is to use a covering of all the edges of $K_n$ with $\frac{n}{2}$ edge-disjoint paths, and analyse the family of paths made up of four random isomorphic copies of this covering. The use of edge-disjoint paths forces the notion of separation to be equivalent to edges not appearing in the same path in one copy of the covering, this is the case regardless of the type of separation. Therefore any weakly separating system generated by this method is automatically strongly separating. For this reason, attempting to adapt the argument of \cite{B et al} to the weak setting can give no improvement in the bound.

Understanding separating path systems of the complete graph is not only important as complete graphs are an interesting and natural class of graph, but also because (in the weak case) $K_n$ has the highest lower bound known for any graph. This means that complete graphs are a good candidate for an extremal graph for this problem, that is, the graph on $n$ vertices which requires the largest separating path system.

The aim of this paper is to investigate the properties of separating path systems of $K_n$, and provide a construction which giving the following upper bound.
\begin{thm}\label{upper bound}
For $n \geq 44$, there exists a separating path system for $K_n$ with size at most
\begin{equation*}
    \frac{21n+16\log_2n+232}{16}.
\end{equation*}
\end{thm}

A key tool for our methods is the notion of Generator Paths, a special path in $K_n$ with nice properties that allow us to create a separating path system by taking only rotated copies of the generator. We show that existence of such a path gives rise to a separating path system of size $n$. We define these special paths and show that they exist for small values of $n$ as well as all prime values of $n$. This gives the following results.
\begin{prop}\label{n<20}
For all $n\leq 20$ we have that $f(K_n)\leq n$.
\end{prop}
\begin{thm}\label{upper bound p p+1}
Let $p$ be an odd prime, then we have $f(K_p) \leq p$ and $f(K_{p+1}) \leq p+1$.
\end{thm}
Our general upper bound (Theorem \ref{upper bound}) comes from constructing an approximate version of this special Generator Path, and correcting any problems with a small number of additional paths. The bulk of the work is in finding this approximation path, which is the content of Section \ref{sec: P construction}.

Before moving on to the results for the complete graph, we highlight some appealing conjectures and results from the literature. Firstly, both \cite{F-R et al} and \cite{B et al} conjectured that any graph $G$ with $n$ vertices must have a (weakly and strongly respectively) separating path system of size $Cn$ where $C$ is some universal constant. This was shown to be true in both cases through a direct proof by Bonamy, Botler, Dross, Naia, and Skokan \cite{Bonamy et al} with $C=19$. The authors state that the value $C=19$ is likely far from optimal, highlighting the question in \cite{F-R et al} of whether it may actually be the case that $f(G) \leq (1+o(1))n$ for all $G$ on $n$ vertices. If this is the case, then clearly the complete graph is one of the least efficient graphs to separate with paths.

\section{Lower Bound}\label{sec: lower bound}
There are several key observations from the definition that highlight the structure of small separating path systems. Let $\mathcal{S}$ be a separating path system for a graph $G$. Firstly, there can be no more than one edge that is not found in any $P \in \mathcal{S}$. This means our family of paths must cover all but one of the edges of $G$.

Secondly, at most one edge in $G$ can appear in every path of $\mathcal{S}$. If not, say $x,y \in P$ for all $P \in \mathcal{S}$, then there is clearly no $P$ that separates $x$ and $y$, therefore $\mathcal{S}$ cannot be a separating path system for $G$.

Finally, looking at some path $P$ in $\mathcal{S}$, if any two edges appear exclusively in $P$ then they cannot be separated by $\mathcal{S}$. Hence, for any path in $\mathcal{S}$ there is at most one unique edge, that is, at most one edge that does not appear in any other path.

With these observations in mind, we repeat the lower bound argument from \cite{F-R et al} to demonstrate the properties small separating path systems should have.

\begin{lem}\label{lower bound}
{\normalfont(Falgas-Ravry, Kittipassorn, Korándi, Letzter, and Narayanan \cite{F-R et al})}
For the complete graph on $n$ vertices, the minimum size of a path separating system is at least $n-1$. That is, $f(K_n) \geq n-1$.
\end{lem}

\begin{proof}
Let $n \in \mathbb{N}$ and suppose for a contradiction that $\mathcal{S}$ is a separating path system for $K_n$ such that $|\mathcal{S}|\leq n-2$.

Each $P \in \mathcal{S}$ has at most 1 unique edge, therefore at most $n-2$ edges of $K_n$ appear exactly once in $\mathcal{S}$. Recall that there is at most 1 edge that appears exactly 0 times in $\mathcal{S}$. We use this to count the number of edges that appear at least twice in $\mathcal{S}$, the number of such edges is at least
\begin{equation*}
    \binom{n}{2} - (n-2) - 1 = \frac{1}{2}(n-2)(n-1).
\end{equation*}

On the other hand, as the maximum length of any $P \in \mathcal{S}$ is $n-1$, we have that the total number of edges used (with multiplicity) is at most $(n-1)(n-2)$. The number of edges in $K_n$ appearing in $\mathcal{S}$ is at least $\binom{n}{2} - 1$. We can count the edges that appear twice using these values, this is at most
\begin{equation*}
    (n-2)(n-1) - \left(\binom{n}{2} - 1\right) = \frac{1}{2}(n-2)(n-3).
\end{equation*}

This is a contradiction since $\mathcal{S}$ cannot satisfy both conditions, hence $f(K_n) \geq n-1$.
\end{proof}

This proof demonstrates the importance of long paths in separating systems. To get an upper bound matching this lower bound of $n-1$, we must construct families with full length paths. In fact any separating path system for $K_n$ with size $n-1$ must have the following properties:
\begin{itemize}
    \item Each path has length $n-1$,
    \item Every path in the system has one unique edge,
    \item All other edges appear in exactly two paths.
\end{itemize}
With this in mind it is easy to construct separating path systems of size $n-1$ by hand for small values of $n$, showing this lower bound is tight for small $n$.

It should also be noted that the only property of paths used in this proof is the fact that a path contains at most $n-1$ edges. Since this is true of many graphs it can be used in other separation contexts. For example, any tree which is a subgraph of $K_n$ has at most $n-1$ edges, so this lower bound holds for trees in general. In fact, by taking a family of $n-1$ stars, each centred around a different vertex of $K_n$, we create a separating tree system for $K_n$ of size exactly $n-1$. So the lower bound of $n-1$ is tight for separating tree systems.

\section{Upper Bounds}\label{sec: upper bound}
While there has been no direct work on an upper bound for our problem, Balogh, Csaba, Martin, and Pluh\'{a}r gave the following result for the strong version of the problem.

\begin{thm}\label{B et al upper bound}
{\normalfont(Balogh, Csaba, Martin, and Pluh\'{a}r \cite{B et al})}
For $n\geq 10$ there exists a strong separating path system for $K_n$ with size at most $2n+4$.
\end{thm}

This gives an upper bound of $f(K_n) \leq 2n+4$, the proof in this case was probabilistic.

\subsection{Generator Paths}\label{sec: gen path}
Thinking of the vertices of $K_n$ as vertices of a regular polygon, we may take any path $P$ and create a new path $P'$ by rotating each edge of $P$ one vertex clockwise. In this way we can find $n$ paths that are isomorphic copies of each other. If our initial path $P$ has some carefully chosen properties forcing each rotation to share exactly one edge with $P$, then we may have a family of separating paths generated by a single path. This reduces the problem to looking for one path rather than a whole system of paths. We need to be careful when choosing our path so that the rotations overlap each other in the right way.

To set it up, label the vertices $1, \dotsc, n$ and arrange them clockwise on a regular polygon.
\begin{wrapfigure}[12]{l}{0pt}
    \centering
\begin{tikzpicture}
\begin{scope}[rotate=162]
    \tikzstyle{edge} = [draw,very thick,black]

    \foreach \x in {1,...,5}{\node[draw,circle,fill=black,inner sep=1pt] (N\x) at ({-(\x)*360/5}:2cm) {}; \node at ({-(\x)*360/5}:2.3cm) {$\x$};}

    \draw[edge] (N4) -- (N1) -- (N2) -- (N5);
    \node at ({-(1.5)*360/5}:2cm) {$P$};
    
    \draw[edge,red] (N1) -- (N3) -- (N2);
    \draw[edge,red,dashed] (N5) -- (N2);
    \node at ({-(2.5)*360/5}:2cm) {\textcolor{red}{$P'$}};
\end{scope}
\end{tikzpicture}
    \caption{Rotations in $K_5$}
    \label{fig:gen path K5}
\end{wrapfigure}
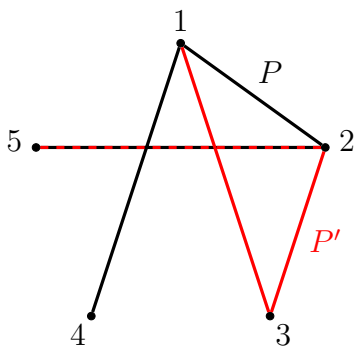
It is helpful to think of our edges by the distance they travel rather than endpoints here. We call an edge at vertex $v$ an \textbf{$x$-type} edge if its other endpoint is $v+x$ for $1 \leq x \leq \lfloor \frac{n}{2} \rfloor$, where all calculations are done modulo $n$. Note that there are precisely $n$ edges in $K_n$ that have type $x$ for $1 \leq x <  \frac{n}{2}$, and in the case where $n$ is even there are precisely $\frac{n}{2}$ edges that have type $\frac{n}{2}$.

In Figure \ref{fig:gen path K5} we have an example of a path $P=(4,1,2,5)$ in $K_5$. This path contains two $2$-type edges, at vertices $4$ and $5$, and one $1$-type edge at vertex $1$. Consider the path $P'=(5,2,3,1)$, it also contains two $2$-type edges, this time at vertices $5$ and $1$, and one $1$-type edge at vertex $2$. This new path is simply $P$ rotated once clockwise. Note that the edge types appear in the same order and at the same distance from each other, this is because edge type is unchanged by rotation. If we take the family consisting of all 5 rotations of $P$, we get a separating path system for $K_5$.

Naturally we wish to pin down the properties a path must have in order to generate a separating path system. But first we clarify some terminology and notation. For any $x$-type edge $e=(v,v+x)$ we call $v$ the \textbf{starting vertex} of $e$. Note that every edge has a unique starting vertex unless it is of type $\frac{n}{2}$, in which case either endpoint can be considered as the starting vertex. We define \textbf{the clockwise distance (on $K_n$)} between vertices $v$ and $u$ to be the value $\min(|v-u|, n-|v-u|)$ and write $cd(v,u)$ for this. Similarly, we say \textbf{the clockwise distance (on $K_n$)} between two edges $e$ and $e'$ to mean the clockwise distance between the starting vertices of $e$ and $e'$ and we write $cd(e,e')$ for this value. In particular, if $e=(v,v+x)$ has type $x$ and $e'=(u,u+y)$ has type $y$, then $cd(e,e')=cd(v,u)=\min(|v-u|, n-|v-u|)$.

\begin{defn}\label{gen path}
A path $P$ on the complete graph $K_n$ is called a \textbf{generator path} for $n$ if it satisfies the following conditions. For odd $n$:
\begin{description}[style=multiline, labelwidth=1.5cm]
    \item[\namedlabel{GP:1}{(GP1)}] $P$ contains at least one edge of each of the $\frac{n-1}{2}$ types.
    \item[\namedlabel{GP:2}{(GP2)}] There is at most one edge type that appears exactly once, and there are no edge types that appear more than twice in $P$.
    \item[\namedlabel{GP:3}{(GP3)}] Let $e$ and $e'$ be $x$-type edges in $P$, and let $cd(e,e')=d$. If there are two other edges $h,h' \in P$, such that both have type $y \neq x$, then $cd(h,h') \neq d$.
\end{description}
For even $n$:
\begin{description}[style=multiline, labelwidth=1.5cm]
    \item[\namedlabel{EGP:1}{(GP1)}] $P$ contains at least one edge of each of the $\frac{n}{2}$ types.
    \item[\namedlabel{EGP:2}{(GP2)}] There is at most one edge type from $[\frac{n}{2}-1]$ that appears exactly once, and there are no edge types that appear more than twice in $P$.
    \item[\namedlabel{EGP:3}{(GP3)}] Let $e$ and $e'$ be $x$-type edges in $P$, and let $cd(e,e')=d$. If there are two other edges $h,h' \in P$, such that both have type $y \neq x$, then $cd(h,h') \neq d$. Additionally, no pair of edges in $P$ with the same type are at distance $\frac{n}{2}$.
\end{description}
\end{defn}

The next result shows that if we can find a generator path for $n$, then it will give us a separating path system of size $n$.

\begin{lem}\label{rotaion gives sps}
If $P$ is a generator path for $n$, then the family of all $n$ rotations of $P$ is a separating path system for $K_n$. Hence, if such a path $P$ exists then $f(K_n)\leq n$.
\end{lem}

\begin{proof}
Let $P$ be a generator path for $n$, and $\mathcal{S}$ be the family generated by taking all the rotations of $P$. In particular, let $P=P_0$ and $P_i=\{ (u+i,v+i) : (u,v) \in P \}$ be the path generated by rotating $P$ clockwise by $i$. Since $P$ is a path in $K_n$ we have that every $P_i \in \mathcal{S}$ is a path in $K_n$. It is left only to show that every pair of edges in $K_n$ is separated by some path in $\mathcal{S}$.

Consider edges $e=(v,v+x)$ and $e'=(v',v'+y)$ in $K_n$, where $e$ is an $x$-type edge and $e'$ is a $y$-type edge. 

By \ref{GP:1}, $P$ contains at least one edge of type $x$, therefore there is some $i\in[0,n-1]$ such that $e \in P_i$. If $e' \notin P_i$ then the edges are separated by $P_i$, so assume otherwise.

Suppose first that $y=x$. Let $cd(e,e')=d$,  without loss of generality we can assume that $v'=v+d$. Consider the path $P_{i+d}$, where calculations are modulo $n$, we must have $e' \in P_{i+d}$. Note that $v'+d = v +2d \neq v$. Indeed, this is only true when $2d=n$, since $P$ is a generator path we have $d \neq \frac{n}{2}$ by \ref{EGP:3}. Therefore, $e \notin P_{i+d}$ since $P_i$ does not contain any additional $x$-type edges (by \ref{GP:2}). Hence, $e$ and $e'$ are separated by $P_{i+d}$.

Now suppose that $y \neq x$. We first consider the case where $n$ is even and $x=\frac{n}{2}$. If $x=\frac{n}{2}$, then we have that $e \in P_{i+\frac{n}{2}}$. Since $e' \in P_i$ in order for $e'$ to be in $P_{i+\frac{n}{2}}$ we must have that $P_i$ contains another $y$-type edge $h'=(u',u'+y)$, such that $u'+\frac{n}{2}=v'$ or $u'-\frac{n}{2}=v'$. In other words we must have $cd(e',h')=\frac{n}{2}$, this cannot happen by \ref{EGP:3}. Similarly, if $y=\frac{n}{2}$ then $e$ and $e'$ are separated by $\mathcal{S}$.

Finally we suppose that $y \neq x$ and $x,y \neq \frac{n}{2}$. Then by \ref{GP:1}, $P$ contains two edges of one of the types $x$ or $y$. Without loss of generality assume that $P$ contains two $x$-type edges. Let $h=(u,u+x) \in P_i$ be the other $x$-type edge, and let $cd(e,h)=d$. We have that $e \in P_j$ for one of $j=i+d$ or $j=i-d$. The only way for $e' \in P_j$ is if there is another $y$-type edge $h'=(u',u'+y) \in P_i$ such that $cd(e',h')=d$. This clearly cannot happen by \ref{GP:3}. So we have that $e$ and $e'$ are separated by $P_j$.
\end{proof}

This makes finding separating path systems for small values of $n$ fairly straightforward. In fact we can find generator paths for all $n\leq 20$ by hand. The generator paths for $n=2,3,4$ are trivial, the generator paths $P(n)$ for other values of $n \leq 20$ are given below:\\
$P(5)=(1,3,2,5)$,\\ 
$P(6)=(1,5,4,3,6)$,\\
$P(7)=(1,2,3,5,7,4)$,\\
$P(8)=(1,3,5,2,6,7,8)$,\\
$P(9)=(1,5,9,3,4,6,8,2)$, \\
$P(10)=(1,4,7,6,5,9,3,8,10)$,\\
$P(11)=(1,3,5,10,4,11,7,8,9,6)$,\\
$P(12)=(1,2,11,9,10,3,7,4,8,6,12,5)$,\\
$P(13)=(1,3,4,13,11,6,10,7,12,5,8,9)$,\\
$P(14)=(1,3,6,9,10,11,2,7,13,5,12,8,4)$,\\
$P(15)=(1,14,15,5,10,3,12,6,9,13,2,4,11,8,7)$,\\
$P(16)=(1,11,13,15,14,3,8,12,16,9,2,10,7,4,5)$,\\
$P(17)=(1,3,5,16,10,11,12,9,6,15,7,14,4,17,13,8)$,\\
$P(18)=(1,15,10,5,13,3,12,9,6,7,8,2,14,16,18,11,4)$,\\
$P(19)=(1,3,5,18,12,11,10,13,16,7,17,6,14,9,4,19,15,8)$,\\
$P(20)=(1,5,10,15,18,8,17,6,20,14,7,19,2,4,16,9,13,12,11)$.
\begin{repprop}{n<20}
For all $n\leq 20$ we have that $f(K_n)\leq n$.
\end{repprop}

As well as these small examples, we can also find generator paths for prime values of $n$. The construction uses the properties of primitive roots, an integer $g$ is a primitive root modulo $n$ if for every integer $h$ which is co-prime to $n$ there is some integer $i$ such that $h \equiv g^i \mod n$.

\begin{lem}\label{gen path for p}
There exists a generator path for $n$ whenever $n$ is an odd prime.
\end{lem}
\begin{proof}
Let $p$ be an odd prime, then it is known that there exists a primitive root $g$ modulo $p$. We can therefore write every integer in $[p-1]$ in the form $g^i$ for $i \in [p-1]$.

Consider the path on $K_p$ given by $P=(p,g,g+g^2,g+g^2+g^3, \dotsc, \sum_{i=1}^{p-2}g^i)$, which is the path starting at vertex $p$ and taking a $g$ length edge, followed by a $g^2$ length edge, followed by a $g^3$ length edge, and so on until there are $p-2$ edges in the path. Here we use `length' to avoid confusion, an $x$-type edge may have length $x$ or $n-x$. We claim that this path is a generator path for $n=p$.

We must check that $P$ is indeed a path in $K_p$, and that it satisfies \ref{GP:1}, \ref{GP:2}, and \ref{GP:3}.

First we show that $P$ is indeed a path. To see this note that $P$ is a path as long as none of the vertices $p,g,g+g^2,g+g^2+g^3, \dotsc, \sum_{i=1}^{j}g^i ,\dotsc, \sum_{i=1}^{p-2}g^i$ are congruent modulo $p$.

Suppose that
\begin{equation*}
    \sum_{i=1}^{m}g^i \equiv \sum_{i=1}^{j}g^i  \mod p
\end{equation*}
for some $m,j \in [p-2]$. Then we must have that
\begin{equation*}
    \frac{g^{m+1}-g}{g-1} \equiv \frac{g^{j+1}-g}{g-1} \mod p
\end{equation*}
and hence
\begin{equation*}
    g^{m-j} \equiv 1 \mod p.
\end{equation*}
Since $g$ is a primitive root we know that $g^{p-1} \equiv 1 \mod p$ and $g^i  \not\equiv 1 \mod p$ whenever $i < p-1$. Therefore we must have
\begin{equation*}
    m \equiv j \mod p-1.
\end{equation*}
Clearly this means $m=j$ since $m,j \in [p-2]$. Now we must also check that the vertex $p$ is distinct from the others. For a contradiction suppose that
\begin{equation*}
    \sum_{i=1}^{m}g^i \equiv p  \mod p
\end{equation*}
for some $m \in [p-2]$. This gives
\begin{equation*}
    g^m \equiv 1 \mod p
\end{equation*}
so we must have $m \equiv 0 \mod p-1$, a contradiction of $m\in [p-2]$.

We conclude that $P$ is indeed a path in $K_p$. It remains to show that $P$ satisfies the three conditions.

Let $k = \frac{p-1}{2}$ and note that since $g$ is a primitive root and $p$ is prime, we have that $g^k \equiv -1 \mod p$ and $g^i  \not\equiv -1 \mod p$ whenever $i < k$. This means that $g^i + g^{k+i} \equiv 0 \mod p$. In other words an edge in $P$ with length $g^i$ and an edge with length $g^{k+i}$ have the same edge type.

Observe that the first $k$ edges of $P$ each have unique edge type, and the $(k+i)$th edge of $P$ has the same type as edge $i$. Thus $P$ contains two edges of every type except type $1$, since the $k$th edge in $P$ is the unique $1$-type edge. Hence, $P$ satisfies \ref{GP:1} and \ref{GP:2}.

For any pair of same type edges in $P$ one edge will be length $g^j$ and the other will be length $g^{k+j}$ for $j \in [k]$.  Since $g^i \not \equiv g^{k+i} \mod p$ this means we have one of two cases, either
\begin{enumerate}
    \item the starting vertex of the $g^j$ edge will be $\sum_{i=1}^{j-1}g^i$ and the starting vertex of the $g^{k+j}$ edge will be $\sum_{i=1}^{k+j}g^i$, or
    \item the starting vertex of the $g^j$ edge will be $\sum_{i=1}^{j}g^i$ and the starting vertex of the $g^{k+j}$ edge will be $\sum_{i=1}^{k+j-1}g^i$.
\end{enumerate}
Note that the clockwise distance between the edges for the first case is given by
\begin{equation*}
    \sum_{i=1}^{k+j}g^i - \sum_{i=1}^{j-1}g^i \equiv \frac{g^{k+j+1}-g^j}{g-1} \mod p \quad \text{ or } \quad -\frac{g^{k+j+1}-g^j}{g-1} \mod p
\end{equation*}
whichever is in $[k]$, and in the second case it is given by
\begin{equation*}
    \sum_{i=1}^{k+j-1}g^i - \sum_{i=1}^{j}g^i \equiv \frac{g^{k+j}-g^{j+1}}{g-1} \mod p \quad \text{ or } \quad -\frac{g^{k+j}-g^{j+1}}{g-1} \mod p.
\end{equation*}
Since we know the additive inverse of $g^j$ is $-g^j \equiv g^{k+j} \mod p$ we can see that the distances in the two cases are equivalent. So we assume we are in case 1.

We must now show that the clockwise distance between pairs of same type edges is not repeated, \ref{GP:3}.

Consider edges in $P$ given by $g^j$,$g^{k+j}$,$g^m$, and $g^{k+m}$.
Suppose that
\begin{equation*}
    \frac{g^{k+j+1}-g^j}{g-1} \equiv \frac{g^{k+m+1}-g^m}{g-1} \mod p.
\end{equation*}
Then we have that 
\begin{equation*}
    g^{k+j+1}(1-g^{m-j}) \equiv g^j(1-g^{m-j}) \mod p
\end{equation*}
and hence
\begin{equation*}
    g^{k+1} \equiv 1 \mod p.
\end{equation*}
This is a contradiction as we know that $g^{p-1} \equiv 1 \mod p$ and $g^i  \not\equiv 1 \mod p$ whenever $i < p-1$. The proof when we take one or both of the clockwise distances to be $-\frac{g^{k+i+1}-g^ji}{g-1} \mod p$ is equivalent.

Therefore $P$ is indeed a generator path for $n=p$.
\end{proof}

This gives us an upper bound for $f(K_n)$ when $n$ is a prime number.
\begin{cor}\label{upper bound p}
We have $f(K_p) \leq p$ whenever $p$ is prime.
\end{cor}
We can also use the structure of the generator path in Lemma \ref{gen path for p} along with the properties of primes to give an upper bound for $n=p+1$.

\begin{thm}\label{upper bound p+1}
Let $p$ be an odd prime, then we have $f(K_{p+1}) \leq p+1$.
\end{thm}
\begin{proof}
Let $K'$ be the complete graph on $p$ vertices given by removing the vertex $p+1$ from $K_{p+1}$. Let $P$ be the generator path for $p$ given in the proof of Lemma \ref{gen path for p}, and let $\mathcal{P}=\{P_i : 0 \leq i \leq p-1\}$ be the family of rotations of $P$, where $P_i$ is $P$ rotated clockwise by $i$. Recall that the edge $h=(\sum_{i=1}^{k-1}g^i,\sum_{i=1}^{k}g^i)$ where $k=\frac{p-1}{2}$ is the unique $1$-type edge in $P$.

Let $T = \{(1,2),(2,3),\dotsc,(p-1,p),(p,1)\}$ be the set of all $1$-type edges in $K'$, and note that $T \setminus \{h\}$ is a path in $K'$ and also in $K_{p+1}$.

Let $P'_i = P_i \cup \{(i,p+1)\}$ for every $1 \leq i \leq p-1$. Since each $P_i$ has the vertex $i$ as endpoint, each $P'_i$ is a path in $K_{p+1}$ and contains $P_i$ as a sub-path.

Define $\mathcal{S} = \{P_0,P'_1,P'_2,\dotsc,P'_{p-1},T \setminus \{h\}\}$. We claim that $\mathcal{S}$ is a separating path system for $K_{p+1}$.

Clearly all elements of $\mathcal{S}$ are paths in $K_{p+1}$, so it remains to check that any two edges are separated by some path in $\mathcal{S}$. Let $e,e' \in E(K_{p+1})$, suppose first that $e,e' \in E(K')$. Then $e$ and $e'$ are separated by $\{P_0,P'_1,P'_2,\dotsc,P'_{p-1}\}$ since we know they are separated by $\mathcal{P}$.

Suppose instead that $e'\in E(K')$ and $e \notin E(K')$. Then $e$ must be of the form $(v,p+1)$ for some $v \in [p]$. Let $x\in [\frac{p-1}{2}]$ be the edge type of $e'$ in $K'$, and note that if $x \neq 1$ there exists $0 \leq i,j \leq p-1$ such that $e' \in P_i,P_j$. Therefore we have $e' \in P'_i,P'_j$. Clearly $e=(v,p+1)$ cannot be in both of these paths, therefore the edges $e$ and $e'$ must be separated by $\mathcal{S}$. Now, if $x=1$ note that either $e' \in P_0$ or $e' \in T \setminus \{h\}$, and since $e$ cannot be in either of these paths, the edges are again separated by $\mathcal{S}$.

Finally, suppose $e,e' \notin E(K')$. Then we can write them in the form $e=(v,p+1)$ and $e'=(u,p+1)$ where $u,v \in [p]$. Since $u \neq v$ we must have that at least one of $u$ and $v$ lies in $[p-1]$, without loss of generality assume $v \in [p-1]$. Then we have that $e \in P'_v$, and clearly $u \notin P'_v$. Hence the edges are separated by $\mathcal{S}$.

\end{proof}
Together, Corollary \ref{upper bound p} and Theorem \ref{upper bound p+1} give Theorem \ref{upper bound p p+1}. Based on the small cases along with Lemma \ref{gen path for p}, we suspect that a generator path exists for every $n$, and in particular that the lower bound in Theorem \ref{lower bound} is close to the true value of $f(K_n)$ for all $n$. It would be useful to know for certain whether or not generator paths exist in general.
\begin{question}
For which values of $n \in \mathbb{N}$ do generator paths exist?
\end{question}
Our motivation for this question comes from separating path systems, but the question of existence of paths containing specific arrangements of edge types seems tricky and interesting in its own right. For instance, McKay and Peters investigate which multisets of edge types are realisable as a path in \cite{McKay Peters}.

Furthermore, if $P$ is a generator path such that each edge type appears exactly twice (except $\frac{n}{2}$ in the even case), then the separating path system produced by rotations of $P$ is in fact strongly separating. This follows from an easy variant of the proof of Lemma \ref{rotaion gives sps}. This would also give an upper bound of $n$ for the strong problem. In fact, $P(12)$ and $P(15)$ above are both generators of strongly separating path systems.
\begin{question}
For which values of $n \in \mathbb{N}$ do generator paths with two edges of each type from $[\frac{n-1}{2}]$ exist?
\end{question}

It is also worth noting that the definition given in \ref{gen path} is a little more strict than needed for generating a separating path system. For $n$ odd, let $P$ be a path in $K_n$ that contains $3$ edges of type $x$ such that $P \setminus \{e\}$ (which is not necessarily a path) satisfies \ref{GP:1}, \ref{GP:2}, and \ref{GP:3}, where $e$ is one of the $x$-types in $P$. Then the rotations of $P$ give a separating path system for $K_n$ as long as $cd(e,e')$ and $cd(e,e'')$ are not both equal to $\frac{n}{3}$, where $e'$ and $e''$ are the other $x$-type edges in $P$. There is also an equivalent condition for even values of $n$, with more care taken when $x=\frac{n}{2}$. Since the conditions for these paths are slightly more relaxed they are possibly easier to find, but we do not make use of them in this paper so we keep to the simpler definition here. 

\subsection{Main Construction}\label{sec: approx paths}
Since the conditions for a generator path are all based on edge types, any path $P$ which is not a generator path fails the conditions for at least one type. In other words, we can separate the edge types of $P$ into two classifications, those that follow \ref{GP:1}, \ref{GP:2}, and \ref{GP:3}, $F$, and those that do not, $D$. Formally we have the following definition.

\begin{defn}\label{def: f sep}
A set of edges $A$ is an $F$-separator for $K_n$ if we can partition the edge types of $K_n$ into sets $F$ and $D$ such that the following holds.
\begin{enumerate}
    \item $A$ contains at least one $y$-type edge for every $y \in F \cup D$.
    \item $A$ contains exactly two $x$-type edges for every $x \in F$.
    \item $A_F$ satisfies \ref{GP:3}, where $A_F=\{e\in A: \text{the edge type of } e \text{ is in } F\}$.
\end{enumerate}
If $A$ is a path in $K_n$, we call it an $F$-separator path. Whenever we have an $F$-separator we will use $D:=[\frac{n}{2}] \setminus F$ to mean the other half of the partition.
\end{defn}

Note that a generator path is an $F$-separator path, where $F= [\frac{n-1}{2}]\setminus\{x\}$ with $x$ the edge type that appears exactly once (of which there is at most one).

It is plain to see by the argument in the proof of Lemma \ref{rotaion gives sps}, the family of paths constructed by taking all the rotations of an $F$-separator path $P$ will separate all edges with type in $F$ from each other. This means if we can find some $P$ which is an $F$-separator path for some large set $F$, then we can find a family of $n$ paths that separates most of the edges in $K_n$. We are then able to use a number of `fixing' paths that will separate edges with type in $D$. The family containing the two types of path gives a separating path system of $K_n$.

Our aim for the rest of this section is as follows. We first define these fixing paths, giving a separating path system based on any path (Theorem \ref{rotations plus fixings}). We then construct an $F$-separator path $P$ with large $F$ (and hence small $D$) to use as the base (Theorem \ref{existence of P}). Our construction only works when $n \equiv 3 \mod 6$ or $n \equiv 5 \mod 6$. The final step is to use the construction of $P$ and the family based on it to give separating path systems for other values of $n$ (Theorem \ref{upper bound}).

Note that the rotation construction method naturally forces edges of the same type to be separated from each other. Indeed, looking at the proof of Lemma \ref{rotaion gives sps} we see that for the case $x=y$ the only condition needed was that $cd(e,e')\neq \frac{n}{2}$. This is always the case when $n$ is odd since $\frac{n}{2}$ is not an integer. So as long as there is no edge type appearing more than twice in $P$, and $n$ is odd, we have that all edges of the same type are separated from each other. In general, if $P$ contains exactly $m$ edges of the same type $e_1,\dotsc,e_m$, then they are separated from each other by the rotations of $P$ unless $cd(e_i,e_{i+1})=\frac{n}{m}$ for all $i \in [m-1]$. We say that any $P$ with such edges of type $x$ has \textbf{equally spaced $x$-type edges}. This leads us to Theorem \ref{rotations plus fixings}, but before we state this we need the following result.

\begin{lem}\label{Q paths}
For all odd $n$ and each edge type $x$ there exist two paths in $K_n$, $Q_x$ and $Q_x'$, such that $Q_x\cup Q_x'$ covers all $x$-type edges in $K_n$ and all edges in $Q_x\cup Q_x'$ have type from $\{1,x\}$.
\end{lem}
\begin{proof}
Let $f$ be the highest common factor of $n$ and $x$, and let $a\in \mathbb{N}$ be such that $af=n$.

Consider the subgraph of $K_n$ containing only $x$-type edges, we have exactly $f$ isomorphic cycles of length $a$. Let $C_1,C_2,\dotsc,C_f$ be the cycles labelled so that a vertex $i$ will be contained in cycle $C_{i \mod f}$. We now describe a path in $K_n$ using sections of each $C_i$ and some linking $1$-type edges.

The path $Q_x$ starts at vertex $1$, and follows $C_1$ for $a-1$ edges, before moving to $C_2$ with a $1$-type edge. It then follows $C_2$ for $a-1$ edges. This continues until $a-1$ edges from each of the cycles $C_1,C_2,\dotsc,C_f$ have been followed. Formally, set $v_1=1$ and recursively define $v_i'=v_i -x \mod n$ and $v_{i+1}=v_i'+1$ for $i\in [f]$. Note that for each $i\in [f]$ the edge $(v_i,v_i') \in C_i$ and therefore has type $x$. Also note that the edges $(v_i',v_{i+1})$ are $1$-type edges for all $i\in [f]$. We define the path $Q_x$ as follows,
\begin{equation*}
    Q_x := \left(\bigcup_{i=1}^{f} C_i \setminus \{(v_i,v_i')\}\right) \cup \left\{(v_i',v_{i+1}) : i \in [f-1]\right\}.
\end{equation*}
Note this is indeed a path since the cycles $C_1,C_2,\dotsc,C_f$ are pairwise vertex disjoint, and each edge $(v_i',v_{i+1})$ links $C_i$ to $C_{i+1}$.

Then we define $Q_x'$ to be
\begin{equation*}
    Q_x' := \{(v_i,v_i') : i \in [f]\} \cup \{(v_i',v_{i+1}) : i \in [f-1]\}.
\end{equation*}
Thus all the $x$-type edges are covered in two paths using only $x$-type and $1$-type edges.
\end{proof}

Now we can give our result.

\begin{thm}\label{rotations plus fixings}
Let $n \in \mathbb{N}$ be odd, and $P$ an $F$-separator path for $K_n$ with no equally spaced $x$-type edges. Then $\mathcal{P}\cup\mathcal{D}$ is a separating path system for $K_n$ where $\mathcal{D} = \{Q_x,Q_x' :x \in D\cup \{1\} \}$ (with $Q_x$ and $Q_x'$ as in Lemma \ref{Q paths} and $D =[\frac{n}{2}] \setminus F$), and $\mathcal{P}$ is the family of $n$ rotations of $P$. In particular $f(K_n) \leq n +2|D\cup \{1\}|$.
\end{thm}
\begin{proof}
Consider a pair of edges $e,e' \in E(K_n)$. Suppose first that $e$ and $e'$ have the same type, $x$. Then $\mathcal{P}$ separates $e$ and $e'$. Indeed, if $e,e' \in P_i$ (the rotation of $P$ by $i$ vertices clockwise), and $cd(e,e')=d$ then one of $P_{i+d},P_{i+2d},P_{i+3d},\dotsc$ contains $e$ but not $e'$. Otherwise we would have a collection of $\frac{n}{d}$ edges $e_1,\dotsc,e_{\frac{n}{d}}$ (which includes $e$ and $e'$), all with type $x$ and such that $cd(e_i,e_{i+1})=d$. This cannot be, since we have no equally spaced $x$-type edges in $P$.

Suppose instead that $e$ has type $x$ and $e'$ has type $y$. If $x,y \in F$ then $e$ and $e'$ are separated by $\mathcal{P}$ since $P$ is an $F$-separator path. The paths $Q_1,Q_1'\in \mathcal{D}$ contain only $1$-type edges and together cover all $1$-types in $K_n$. Therefore if $x=1$ then one of these two paths separates $e$ and $e'$. So assume $x,y \neq 1$, and that $x \notin F$. Then $Q_x$ and $Q_x'$ contain only $x$-type and $1$-type edges and together cover all $x$-types in $K_n$. One of these paths separates $e$ and $e'$.
\end{proof}

To prove Theorem \ref{upper bound} it is left to find some path $P$ with small $D$ and no equally spaced edges of the same type.

\begin{thm}\label{existence of P}
When $n \equiv 3 \mod 6$ or $n \equiv 5 \mod 6$, there is an $F$-separator path $P$ for $K_n$, with no equally spaced $x$-type edges and with $|D\cup\{1\}|\leq \frac{1}{32}(5n+16\log_2n+167)$. Where $D =[\frac{n}{2}] \setminus F$.
\end{thm}

The construction of this path is rather involved and makes up the bulk of the argument, the proof can be found in Section \ref{sec: P construction}. We first show how to use this result to get separating path systems for all sufficiently large values of $n$. The methods for extending paths that are used here are similar to those used in the proof of Theorem \ref{upper bound p+1}.

\begin{repthm}{upper bound}
For $n \geq 44$, there exists a separating path system for $K_n$ with size at most
\begin{equation*}
    \frac{21n+16\log_2n+232}{16}.
\end{equation*}
\end{repthm}
\begin{proof}
\textbf{Case 1:} $n$ is odd and $\frac{n-1}{2}$ is not a multiple of 3.\\
Let $P$ be the path from Theorem \ref{existence of P}, and $D$ its associated set of badly behaved edge types. Then by \ref{rotations plus fixings} there is a separating path system for $K_n$ with size at most
\begin{equation*}
    n + 2\cdot\frac{5n+16\log_2n+167}{32}.
\end{equation*}

\textbf{Case 2:} $n-1$ is odd and $\frac{n-2}{2}$ is not a multiple of 3.\\
Consider the complete graph on $n-1$ vertices formed by removing vertex $v$ from $K_n$.
Then $K_{n-1}$ fits the conditions for Case 1, let $P$ and $D$ be as in Case 1. Use \ref{rotations plus fixings} to give a separating path system for $K_{n-1}$ of the form $\mathcal{P}\cup\mathcal{D}$, where $\mathcal{P}$ is the rotations of $P$ and $\mathcal{D} = \{Q_x,Q_x' : x \in D\cup\{1\} \}$. Then $|\mathcal{P}\cup\mathcal{D}| \leq n-1 + 2|D\cup\{1\}|$ where
\begin{equation*}
    |D\cup\{1\}| \leq \frac{5(n-1)+16\log_2(n-1)+167}{32}.
\end{equation*}

Let $P_i$ denote the rotation of $P$ by $i$ vertices clockwise on $K_{n-1}$, and similarly $w_i$ for the vertex of $K_{n-1}$ which is $i$ clockwise vertices on from $w$. Let $u$ be an endpoint of the path $P=P_0$. Then the family $\mathcal{P'}\cup\mathcal{D}$ is a separating path system for $K_n$ where $\mathcal{P'}= \{ P_i \cup \{(u_i,v)\} : i \in [0,n-2]\}$.

Indeed, any pair of edges from $E(K_{n-1})$ are separated since they are separated by $\mathcal{P}\cup\mathcal{D}$. So let $e=(w,v)$ for some $w \in V(K_{n-1})$, and consider any edge $e' \in E(K_{n-1})$. The edge $e$ only appears in one path from $\mathcal{P'}\cup\mathcal{D}$, namely the path $P_i\cup \{(u_i,v)\}$ where $u_i=w$. Therefore if $e'$ appears in any other path, then we are done. Clearly this is the case if $e'$ has type $x\in D$. Suppose that $e'$ has type $x \in F$, then by definition $P$ must have two $x$-type edges, therefore $e'$ appears in two rotations of $P$.

Finally if $e'$ is also an edge at $v$, then clearly it is not contained in the path $P_i\cup \{(u_i,v)\}$, and the two edges must be separated.

Note that $|\mathcal{P'}|=|\mathcal{P}|=n-1$, and $|\mathcal{D}|=2|D\cup\{1\}|$.

\textbf{Case 3:} $n$ is odd and $\frac{n-1}{2}$ is a multiple of 3.\\
Consider the complete graph on $n-2$ vertices given by removing vertices $v$ and $v'$ from $K_n$. Clearly $K_{n-2}$ satisfies the conditions for Case 1, so let $\mathcal{P}\cup\mathcal{D}$ be the corresponding separating path system for $K_{n-2}$ given by Theorem \ref{rotations plus fixings}. We will take each path from $\mathcal{P}$ and adapt it to a path on $K_n$, we will also add some edge types to the set $D$ and create additional paths for our family this way.

Let $u$ be any endpoint of $P$. Select some edge $g \in P$ such that if $x_g$ is the edge type of $g$, then $P$ contains at least one other edge with type $x_g$. Set $g=(p,q)$ for $p,q \in V(K_{n-2})$, and let $g_i = (p_i,q_i)$ be the rotation of $g$ in $K_{n-2}$ by $i$ vertices. We fix the new paths $P'_i=(P_i \setminus \{g_i\}) \cup \{(p_i,v),(v,q_i),(u_i,v')\}$, and let $\mathcal{P'}=\{P'_i : i \in [0,n-3] \}$.

If there are edges $h, h' \in P$ with the same edge type $x_h \in F$, such that $cd(h,h')=x_g$ on $K_{n-2}$, then define $D'=D \cup \{x_g, x_h\}$. Note that there can be at most one such pair since $x_h \in F$. Otherwise define $D'=D \cup \{x_g\}$. Let $\mathcal{D'}= \{Q_x,Q_x' : x \in D'\cup\{1\} \}$ be the family of fixing paths on each type in $D'$ (such paths described in \ref{Q paths}). Then $\mathcal{P'}\cup\mathcal{D'}$ is a separating path system for $K_n$.

Let $e,e'$ be any two edges in $K_n$. If $e,e' \in E(K_{n-2})\setminus\{g_i : i \in [n-2]\}$ then the edges are separated by $\mathcal{P}\cup\mathcal{D}$ and hence by $\mathcal{P'}\cup\mathcal{D'}$. Suppose $e'=g_i$ for some $i \in [n-2]$, and let $x_e$ be the edge type of $e$. Since $x_g \in D'$ we have that the path $Q \in \{Q_{x_g}, Q_{x_g}'\} \subseteq \mathcal{D'}$ contains the edge $g_i$. The path $Q$ separates $g_i$ from $e$ unless $x_e=x_g$ or $x_e=1$. If $x_e=1$ then one of $Q_1,Q_1' \in \mathcal{D'}$ separates the pair. If $x_e=x_g$ then, since there is still an $x_g$-type edge in $P'$ and there are no equally spaced edge types, $\mathcal{P'}$ separates the pair.

Suppose then that $e=(v,w)$ for some $w \in V(K_{n-2})$, and $e' \in E(K_{n-2})$, then the edges are separated if $e'$ is in some $\mathcal{D'}$ path, otherwise there is some other edge in $P_i$ with the same type as $e'$. The distance between these edges cannot be the same as the distance between the two edges at $v$ since $x_h \in D'$, therefore the edges must be separated by the rotations (as in \ref{rotaion gives sps}). If $e=(v,w)$ and $e'=(v,w')$, then let $i$ be such that $w_i=w'$. We have that if $e,e' \in P'_j$ then $e' \in P_{j+i}$, the only other edge at $v$ in  $P_{j+i}$ is $(v,w'_i)$. Therefore the edges are separated. Finally, if $e=(v',w)$ then note that it is the unique edge a path in $\mathcal{P'}$, since every other edge type either appears twice in $P'$ or in $D'$, there are no other unique edges. Therefore the edges are separated.

Note that $|\mathcal{P'}|=|\mathcal{P}|=n-2$, and $|\mathcal{D'}|=2|D'\cup\{1\}|=2(|D\cup\{1\}|+2)$.

\textbf{Case 4:} $n-1$ is odd and $\frac{n-2}{2}$ is a multiple of 3.\\
This case is very similar to Case 3. We consider the family $\mathcal{P}\cup\mathcal{D}$ from \ref{rotations plus fixings} on the $K_{n-3}$ obtained by removing the vertices $v,v'$ and $\bar{v}$ from $K_n$. Then we adapt the paths $P$ and add to the set $D$ to obtain a new family.

Let $u$ be any endpoint of $P$, and let $g=(p,q),g'=(p',q') \in P$ be edges with type $x_g$ and $x_g'$ respectively such that $x_g \neq x_g'$ and $P\setminus\{g,g'\}$ still contains an edge of type $x_g$ and edge of type $x_g'$. Set the new paths to be $P'_i= (P_i \setminus \{g_i,g_i'\}) \cup \{(p_i,v),(v,q_i),(p'_i,v'),(v',q'_i),(w,\bar{v})\}$ and $\mathcal{P'}=\{P'_i : i \in [0,n-4]\}$.

If there are edges $h, h' \in P$ both with edge type $x_h\in F$ such that the clockwise distance between them on $K_{n-3}$ is equal to $x_g$, then set $D' = D \cup \{x_g,x_g',x_h \}$. If there is also a pair of edges with type $x_h'\in F$ with clockwise distance equal to $x_g'$ then set $D' = D \cup \{x_g,x_g',x_h,x_h' \}$. Note that there can be at most one of each since $x_h,x_h' \in F$. Otherwise, set $D' = D \cup \{x_g,x_g'\}$. Then by the same reasoning as for Case 3, the family $\mathcal{P'}\cup\mathcal{D'}$ is a separating path system for $K_n$.

Note that $|\mathcal{P'}|=|\mathcal{P}|=n-3$, and $|\mathcal{D'}|=2|D'\cup\{1\}|=2(|D\cup\{1\}|+4)$.
\end{proof}

\subsection{Proof of Theorem \ref{existence of P}}\label{sec: P construction}
The aim now is to construct an $F$-separator path with many edge types in $F$. We will start by defining three sets of edges ($M_0$, $R$, and $B$) which together form a linear forest. This linear forest will almost follow \ref{GP:1}, \ref{GP:2}, and \ref{GP:3}. We will then add certain joining edges to connect our forest into a single path. The types associated to the special edges in the linear forest will end up in $F$ and the types associated to the joining edges will end up in $D$.

\textbf{Step 1: Defining the linear forest}

Let $n \in \mathbb{N}$ be such that $n \equiv 3 \mod 6$ or $n \equiv 5 \mod 6$. For the remainder of this section we label the vertices of $K_n$ slightly differently for ease of notation. First label any vertex $0$, from here label vertices clockwise following the order $1,2,\dotsc,\frac{n-1}{2},-\frac{n-1}{2},-(\frac{n-1}{2}-1),\dotsc,-2,-1$. Now we define three sets of edges from $K_n$ which we will combine to make an approximation of a generator path.

First we define a set containing one edge of each type,
\begin{equation*}
    M_0 = \left\{ (-i,i) : i \in \left[\frac{n-1}{2}\right] \right\}.
\end{equation*}
Note that $M_0$ is a maximal matching in $K_n$ and that the vertex $0$ is the only vertex which is not an endpoint of an edge in $M_0$. We also use $M_k$ to denote the rotation of $M_0$ in which vertex $k$ has no incident edges.

Next we define an edge set containing only the largest edge types. Let $R=R_1 \cup R_2$ where
\begin{equation*}
    R_1 = \left\{ \left(1,-\frac{n-3}{2}\right),\left(-1,-\frac{n-1}{2}\right) \right\},
\end{equation*}
and
\begin{equation*}
    R_2 = \left\{ \left(-3-2k, \frac{n-1}{2}-k\right) :
    0 \leq k \leq r-1 \right\},
\end{equation*}
where $r=\frac{n-7}{4}$ for when $\frac{n-1}{2}$ is odd, and $r=\frac{n-9}{4}$ otherwise. 

Note that $R_1$ consists of a $\frac{n-1}{2}$-type edge and a $\frac{n-3}{2}$-type edge. Further, note that $(-3,\frac{n-1}{2})$ has type $\frac{n-5}{2}$, and that if an edge $(-3-2k,\frac{n-1}{2}-k)$ is $x$-type, then $(-3-2(k+1),\frac{n-1}{2}-(k+1))$ is an $(x-1)$-type edge. Thus, $R$ contains the largest $r+2$ edge types.

Define $r(i) = \frac{n+2+i}{2}$ for every $i = -3-2k$ where $0 \leq k \leq r-1$, and $r(i) = -\frac{n-2-i}{2}$ for $i=-1,1$. Then we can write each edge in $R$ as $(i,r(i))$. Similarly let $r^{-1}(i)=2i-n-2$ for $i=\frac{n-1}{2}-k$ where $0 \leq k \leq r-1$, and $r^{-1}(i)= 2i+n-2$ for $i=-\frac{n-3}{2},-\frac{n-1}{2}$. Then we can also write $R$ edges in the form $(r^{-1}(i),i)$.

The contents of the final edge set depend on the edges in $R$ as well as the properties of $\frac{n-1}{2}$. We define this set so that it continues with the large edge types roughly where $R$ left off, containing approximately types $\frac{n}{4}$ down to $\frac{n}{8}$. Choose $i_b$ even and as large as possible such that $(-i_b,i_b+3)$ is an edge with odd type at most $\frac{n-1}{4}$. Set $b=(-i_b,i_b+3)$ and denote $x_b=2i_b+3$ as the edge type of $b$. Note that $b \in M_{-\frac{n-3}{2}}$ and
\begin{equation*}
    x_b \in \left\{ \left\lfloor\frac{n-1}{4}\right\rfloor, \left\lfloor\frac{n-1}{4}\right\rfloor-1, \left\lfloor\frac{n-1}{4}\right\rfloor-2, \left\lfloor\frac{n-1}{4}\right\rfloor-3 \right\}.
\end{equation*}
Then we can define the final edge set as 
\begin{equation*}
    B = \left\{ \left(-\frac{x_b-3}{2}+2k, \frac{x_b+3}{2}+k\right) : 0 \leq k \leq t-1 \right\},
\end{equation*}
where $t=\frac{x_b+1}{2}$.

Let $b(i) = \frac{3x_b+3+2i}{4}$ for each $i=-\frac{x_b-3}{2}+2k$ where $0 \leq k \leq t-1$. Also, let $b^{-1}(i)= -\frac{3x_b+3-4i}{2}$ for every $i=\frac{x_b+3}{2}+k$ where $0 \leq k \leq t-1$. Then we can write $B$ edges in the form $(i, b(i))$ and $(b^{-1}(i), i)$.

See Figure \ref{fig:L for Kn} for an example of the edge sets $M_0$, $R$, and $B$.

\begin{figure}[h!t]
    \centering
\begin{tikzpicture}
\begin{scope}[rotate=90]

    \tikzstyle{edge} = [draw,thick,-,black]
    \foreach \x in {0,...,17}{\node[draw,circle,fill=black,inner sep=1pt] (N\x) at ({-(\x)*360/35}:4.5cm) {};}
    \foreach \x in {1,...,17}{\node[draw,circle,fill=black,inner sep=1pt] (-N\x) at ({(\x)*360/35}:4.5cm) {};}
 
    \foreach \y in {0,...,17}{\node at ({-(\y)*360/35}:4.8cm) {$\y$};}
    \foreach \y in {1,...,17}{\node at ({(\y)*360/35}:4.9cm) {$-\y$};}

    \foreach \x in {1,...,17}{\draw[edge] (N\x) -- (-N\x);}
    
    \draw[edge,red] (N1) -- (-N16);
    \draw[edge,red] (-N1) -- (-N17);
    \draw[edge,red] (-N3) -- (N17);
    \draw[edge,red] (-N5) -- (N16);
    \draw[edge,red] (-N7) -- (N15);
    \draw[edge,red] (-N9) -- (N14);
    \draw[edge,red] (-N11) -- (N13);
    \draw[edge,red] (-N13) -- (N12);
    \draw[edge,red] (-N15) -- (N11);
    
    \draw[edge,cyan] (-N2) -- (N5);
    \draw[edge,cyan] (N0) -- (N6);
    \draw[edge,cyan] (N2) -- (N7);
    \draw[edge,cyan] (N4) -- (N8);

    \node at (180:2cm) {\textcolor{red}{$R$}};
    \node at (-45:3.6cm) {\textcolor{cyan}{$B$}};
\end{scope}
\end{tikzpicture}
    \caption{The linear forest $L$ for $n=35$}
    \label{fig:L for Kn}
\end{figure}
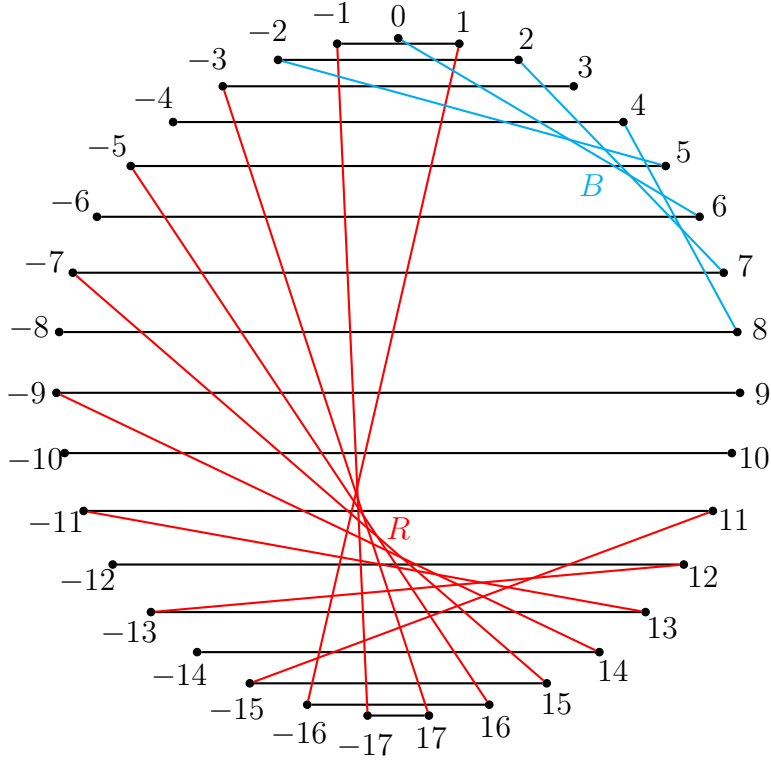

Let $L = M_0 \cup R \cup B$ denote this collection of edges. Our task now is to extend $L$ to a path. In order to do this we must have that $L$ is acyclic, and that the maximum degree of any vertex is $2$.

\begin{claim}
$L$ is a linear forest.
\end{claim}
\begin{proof}
First we consider the degree condition. Observe that $M_0$ is a maximal matching and therefore no $M_0$ edges share vertices. It is clear from the constructions of $R$ that no two $R$ edges share a vertex, similarly no two $B$ edges share a vertex. This means all edges of high degree must be an endpoint of an edge in each of $M_0,R$, and $B$. As usual we write $[a,b] = \{c \in \mathbb{N}: a \leq c \leq b\}$ where $a$ and $b$ are non-negative integers with $a\leq b$. We also write $[-a,b] = \{ c \in \mathbb{Z} : -a \leq c \leq b\}$ where $a$ and $b$ are non-negative integers, and $[-a,-b] = \{ c \in \mathbb{Z} : -a \leq c \leq -b\}$ where $a$ and $b$ are non-negative integers and $a\geq b$.

Let $I_1 = [-2,1] = \{-2,-1,0,1\}$ and $I_2=[-\frac{n-5}{2},-3]$, then each $R$ edge has one vertex in the set $I = I_1 \cup I_2$, moreover this vertex is odd. Let $I_1' = [-\frac{n-1}{2},-\frac{n-3}{2}]$ and $I_2'=[\frac{n+9}{4},\frac{n-1}{2}]$, then each $R$ edge has one vertex in the set $I'=I_1' \cup I_2'$. Similarly, let $J=[-\frac{x_b-3}{2},\frac{x_b+1}{2}]$ and $J'=[\frac{x_b+3}{2},x_b+1]$, each $B$ edge has exactly one vertex in $J$, which is even, and one vertex in $J'$.

Using the fact that all $R$ endpoints in $I$ are odd and every $B$ endpoint in $J$ is even together with the fact that $I\cap J'=\emptyset$, we have that no vertex in $I$ is the endpoint of both a $B$ and an $R$ edge. Then the only candidates for a vertex of high degree must be found in $I'$. Note first that $I_1' \cap (J\cup J')= \emptyset$, so any high degree vertex must come from $I'_2$. The largest vertex which is also an endpoint in $B$ is $x_b+1$. Recall that $x_b \leq \frac{n-1}{4}$, therefore $x_b+1 \leq \frac{n+3}{4} < \frac{n-1}{2}-k$ for all $0 \leq k \leq r-1$. Therefore there are no vertices of degree greater than $2$ in $L$.

It is left to show that $L$ is acyclic. Suppose for a contradiction that $C$ is a cycle in $M_0\cup R$, clearly $C$ must alternate between $M_0$ edges and $R$ edges. Let $e=(-i,i)$ be the edge in $C\cap M_0$ such that $i\in [\frac{n-1}{2}]$ is maximal. Let $r^+,r^- \in C \cap R$ be edges with an endpoint at $i$ and $-i$ respectively. Observe that if $i=\frac{n-1}{2}$ then $r^+=(-3, \frac{n-1}{2})\in R_2$ and so $(-3,3)\in C$. Since $3 \notin I\cup I'$ there is only one edge in $M_0 \cup R$ at $3$, meaning $C$ cannot be a cycle. Further note that the edge $(-\frac{n-3}{2},\frac{n-3}{2})$ is in the same path as $(-\frac{n-1}{2},\frac{n-1}{2})$ by edges in $R_1$, therefore $i < \frac{n-3}{2}$. This means that $r^+,r^- \in R_2$, so let $r^+=(r^{-1}(i),i)$ and $r^-=(-i,r(-i))$. Suppose that $|r^{-1}(i)|<r(-i)$, then after $r^+$ the cycle $C$ must follow an $M_0$ edge to $\ell=|r^{-1}(i)|$ where there must be another $R$ edge $(r^{-1}(\ell),\ell)$ for $C$ to continue. Since $\ell < r(-i)$, the construction of $R_2$ means we have $r^{-1}(\ell) < -i$. Since all $R$ vertices in $I_2$ are negative this means $|r^{-1}(\ell)|>i$. The next edge in $C$ after $(r^{-1}(\ell),\ell)$ must be in $M_0$, so we have that $(r^{-1}(\ell),|r^{-1}(\ell)|) \in C$ where $|r^{-1}(\ell)|>i$ contradicting the maximality of $i$. The case where $|r^{-1}(i)|>r(-i)$ is analogous.

So every cycle must contain a $B$ edge. Again, for a contradiction, let $C$ be a cycle in $L$, then $C$ must alternate between $M_0$ edges and $R\cup B$ edges. Let $(i,b(i))\in B\cap C$ and note that by definition of $B$ we have $|i|<b(i)$. The edge $(-b(i),b(i))$ must appear in $C$. Since $b(i) \in J'$ we have that $-b(i) \notin J\cup J'$, this means there is no $B$ edge at $-b(i)$. So there must be an $R$ edge at $-b(i)$ to continue the cycle, in particular we must have $-b(i) \in I$. In order to have $C$ be a cycle, we must be able to continue in this direction along $C$ and end up back at the vertex $i$.

Let $\ell = r(-b(i))$, then the edge $(-b(i),\ell)$ must be in $C$. Note that $\ell \in I'$ and therefore $\ell > |i|$ since $i \in J$. $C$ must continue after $(-b(i),\ell) \in R$ with the edge $(-\ell,\ell)\in M_0$. Since $-\ell \neq i$ we have not connected $C$. The next edge must be from $R\cup B$, and since $\ell \notin J \cup J'$ this must be another $R$ edge. In particular $(-\ell, r(-\ell))$. Note that as $r(-\ell)\in I'$ we must have that $r(-\ell) > |i|$. Again, the next edge in $C$ must be $(-r(-\ell), r(-\ell))\in M_0$. Since $r(-\ell) > |i|$ we have $-r(-\ell) \neq i$ and we have not completed $C$.

We can continue with this same argument, but at each stage we go through a vertex in $I'$ then to a vertex in $-I'$ via $M_0$, then back to $I'$ with an $R$ edge. This means we can never get back to the vertex $i$ since $|i|$ is smaller than all vertices in $I'$.
\end{proof}

\begin{claim}\label{MRB sat 13}
$L$ satisfies \ref{GP:1} and \ref{GP:3}. Moreover, the sub-forest of $L$ which contains all edges of type matching those in $R\cup B$ satisfies \ref{GP:2}.
\end{claim}

\begin{proof}
For \ref{GP:1}, note that $n$ is odd and $M_0$ contains exactly one edge of every type. Clearly then, $L$ contains at least one edge of each type.

Recall that each edge in $R$ has a different type from $[\lceil \frac{n-1}{4} \rceil,\frac{n-1}{2}]$, and that each $B$ edge has a different type from $[\frac{x_b+3}{2},x_b]$ where $x_b \leq \lfloor \frac{n-1}{4} \rfloor$. This means that there is no edge type that appears more than twice in $L$. Moreover, for any edges $e,e' \in L$ that have the same type, we have $e\in M_0$ and $e' \in R\cup B$. This means the restriction of $L$ to edges with type found in $R\cup B$ contains exactly two edges of each type, hence satisfies \ref{GP:2}.

Since $M_0$ contains exactly one of each edge type, the family of rotations of $M_0$ must cover the edges of $K_n$. In other words, any edge in $R \cup B$ is found in $M_i$ for some $i\in [-\frac{n-1}{2},\frac{n-1}{2}]$. Let $e\in M_0$ and $e'\in R\cup B$ both be $x$-type edges. Then $e'\in M_i$ for some $i$, and since $M_i$ is a rotation of $M_0$ by $|i|$ we must have that $cd(e,e')=|i|$. This means that \ref{GP:3} is equivalent to the conditions
\begin{equation*}
    M_0 \cap (R \cup B) = \emptyset \qquad \text{and} \qquad |(M_i\cup M_{-i}) \cap (L)| \leq 1
\end{equation*}
for every $i \in [\frac{n-1}{2}]$.

We associate to the edges of $K_n$ a number from $[0,\frac{n-1}{2}]$ called the \textit{crossing number}. Each $M_i$ (except $M_0$) contains a single edge $e_i$ at the vertex $0$, define the crossing number of all edges in $M_i$ to be the type of the edge $e_i$. We define the crossing number of edges in $M_0$ to be $0$. Note that by this definition the crossing number of edges in $M_i$ is the same as the crossing number of $M_{-i}$ edges, we also have that any pair $M_i,M_j$ with $|i|\neq|j|$ contain edges with different crossing numbers. Therefore, to show the conditions above hold, it is sufficient to show that no two edges from $R \cup B$ have the same crossing number, and that no edge in $R \cup B$ has crossing number $0$.

Consider the edge $e=(1,-\frac{n-3}{2}) \in R_1$, let $M_\ell$ be the matching containing $e$. Then $M_\ell$ must also contain the edge $(0,-\frac{n-5}{2})$ and hence $e$ has crossing number $\frac{n-5}{2}$. Similarly we see that $(-1,-\frac{n-1}{2})$ has crossing number $\frac{n-1}{2}$. Now, for edges in $R_2$ we first note that $(-3, \frac{n-1}{2})$ has crossing number $m=\frac{n-7}{2}$. Then observe that $(-3-2, \frac{n-1}{2}-1)$ has crossing number $m-3$, and similarly $(-3-2k, \frac{n-1}{2}-k)$ has crossing number $|m-3k|$. Note that $|m-3k|\neq |m-3k'|$ for any $k \neq k'$, so no edges in $R_2$ have the same crossing number. Further note that $m$ is not a multiple of $3$ since $\frac{n-1}{2}$ is not a multiple of $3$, this means no $R_2$ edge has crossing number $0$.

Since $r \leq \frac{n-7}{4}$ we have that $|m-3k|<\frac{n-5}{2}$ for all $0 \leq k \leq r-1$, meaning that there are no edges in $R_2$ with the same crossing number as an edge in $R_1$. In particular this means that no two edges in $R$ have the same crossing number, and no edge in $R$ has a crossing number which is $0$ or a multiple of $3$.

We now consider edges in $B$. Recall that the edge $b=(-\frac{x_b-3}{2}, \frac{x_b+3}{2}) \in M_{-\frac{n-3}{2}}$ and therefore has crossing number $3$. Note also that $(-\frac{x_b-3}{2}+2, \frac{x_b+3}{2}+1)$ has crossing number $6$ and in general the edge $(-\frac{x_b-3}{2}+2k, \frac{x_b+3}{2}+k)$ has crossing number $3(k+1)$ for $0\leq k\leq t-1$. Recall also that $t \leq \frac{n+3}{8}$, therefore $3(k+1) \leq \frac{n-1}{2}$ for all $n\geq 13$. In particular we have that no edges in $B$ have crossing number $0$, and clearly no two edges from $B$ have the same crossing number. We saw earlier that no edges in $R$ have a crossing number which is a multiple of $3$, since edges in $B$ only have crossing numbers which are multiples of $3$ we see that there are no two edges in $R \cup B$ with the same crossing number.
\end{proof}

The above tells us that the only edge types that fail the conditions of Definition \ref{gen path} are those that appear only once (and hence fail \ref{GP:2}). These are exactly the types that do not appear in $R \cup B$. In other words, let
\begin{equation*}
    F_L = \{x: \text{there is an } e\in R\cup B \text{ with edge type } x\}
\end{equation*}
be the set of edge types from edges in $R\cup B$, then $L$ is an $F_L$-separator. Note that $F_L$ is close to the set $[\frac{n}{8},\frac{n-1}{2}]$, possibly with a small number of elements missing.

We now must connect our $L$ into a single path without using too many edges with type in $F_L$.

\textbf{Step 2: Finding connecting edges to join the linear forest into a path}

We can do this by adding edges between endpoints of paths in $L$, but it is in our interests to keep these joining edges as short as possible to avoid using edges from $F_L$. In this step we will find a set of edges $C$ such that $L \cup C$ is a path and where the edge types in $C$ do not overlap too much with those in $F_L$. To do this, we need to know where the endpoints of paths in $L$ appear, and a little more about the behaviour of each path in $L$. We partition the non-zero vertices of $K_n$ into sets depending on their label. First set $T^+ = [x_b+1]$ , $M^+=[x_b+2,\frac{n-1}{2}-r]$ and $U^+=[\frac{n-1}{2}-r+1,\frac{n-1}{2}]$, then $T^-,M^-$ and $U^-$ contain the respective negative vertices (see Figure \ref{fig:segments of Kn}).

\begin{figure}[h!t]
    \centering
\begin{tikzpicture}
\begin{scope}[rotate=90]

    \tikzstyle{edge} = [draw,thick,-,black]
    \foreach \x in {0,...,17}{\node[draw,circle,fill=black,inner sep=1pt] (N\x) at ({-(\x)*360/35}:4.5cm) {};}
    \foreach \x in {1,...,17}{\node[draw,circle,fill=black,inner sep=1pt] (-N\x) at ({(\x)*360/35}:4.5cm) {};}
 
    \foreach \y in {0,...,17}{\node at ({-(\y)*360/35}:4.8cm) {$\y$};}
    \foreach \y in {1,...,17}{\node at ({(\y)*360/35}:4.9cm) {$-\y$};}

    \foreach \x in {1,...,17}{\draw[edge] (N\x) -- (-N\x);}
    
    \draw[edge,red] (N1) -- (-N16);
    \draw[edge,red] (-N1) -- (-N17);
    \draw[edge,red] (-N3) -- (N17);
    \draw[edge,red] (-N5) -- (N16);
    \draw[edge,red] (-N7) -- (N15);
    \draw[edge,red] (-N9) -- (N14);
    \draw[edge,red] (-N11) -- (N13);
    \draw[edge,red] (-N13) -- (N12);
    \draw[edge,red] (-N15) -- (N11);
    
    \draw[edge,cyan] (-N2) -- (N5);
    \draw[edge,cyan] (N0) -- (N6);
    \draw[edge,cyan] (N2) -- (N7);
    \draw[edge,cyan] (N4) -- (N8);

    \draw[edge,gray,dashed] (0:5.25cm) -- (180:5.5cm);
    \draw[edge,gray,dashed,shorten >=-1.5cm, shorten <=-1.5cm] ({-(8.5)*360/35}:4.5cm) -- ({(8.5)*360/35}:4.5cm);
    \draw[edge,gray,dashed,shorten >=-1.7cm, shorten <=-1.7cm] ({-(10.5)*360/35}:4.5cm) -- ({(10.5)*360/35}:4.5cm);

    \node at (45:5.5cm) {$T^-$};\node at (-45:5.5cm) {$T^+$};
    \node at (95:5.5cm) {$M^-$};\node at (-95:5.5cm) {$M^+$};
    \node at (135:5.5cm) {$U^-$};\node at (-135:5.5cm) {$U^+$};

\end{scope}
\end{tikzpicture}
    \caption{The segments of $L$ for $n=35$}
    \label{fig:segments of Kn}
\end{figure}
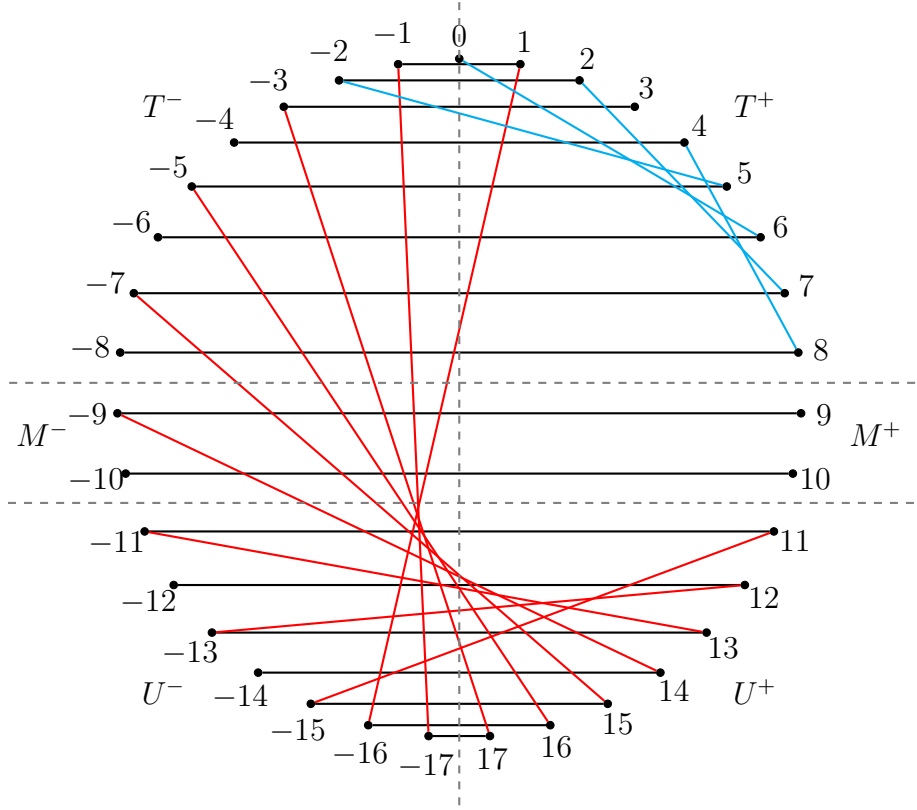

First consider vertices in $U^+$, note that each vertex in $U^+$ is an $I'$ vertex of an $R_2$ edge. Since these vertices also all have an incident $M_0$ edge we conclude that there are no endpoints in $U^+$.

Next $M^+$, since all vertices in $M^+$ are larger than $x_b+1$ and smaller than $\frac{n-1}{2}-r+1$ there can be no incident $B$ or $R$ edges here. Hence every vertex in $M^+$ is an endpoint of a path in $L$. Note that $|M^+|=\frac{n-1}{2}-r-x_b-1$, then due to the values of $x_b$ and $r$, we must have that $1 \leq|M^+|\leq 4$.

For $T^+$ note that every vertex in $J'\subseteq T^+$ is an endpoint of a $B$ edge, therefore the only vertices in $T^+$ with degree $1$ must be in $[\frac{x_b+1}{2}]$. In particular, by the construction of $B$, these endpoints must all be odd vertices since $-\frac{x_b-3}{2}+2k \in J$ is even for $0 \leq k \leq t-1$. Further note that there is an $R_1$ edge at vertex $1$, but this is the only $R$ edge with an endpoint in $T^+$. Together this means all odd vertices in $[3,\frac{x_b-1}{2}]$ are endpoints in $L$, this is $\frac{x_b-3}{4}$ vertices since $(i_b+3)-1=\frac{x_b+1}{2}$ is even.

For the negative side, we know that all odd vertices have an incident $R$ edge, so the only endpoints on this side must be even. For $U^-$ note that $-\frac{n-3}{2}$ and $-\frac{n-1}{2}$ are both vertices with edges from $R_1$, so the endpoints in $U^-$ are all even vertices in $[-\frac{n-5}{2},-\frac{n-1}{2}+r-1]$. For $M^-$ this means that all even vertices are endpoints. In particular each endpoint is from a path consisting of one edge and terminating in $M^+$, there are $\lfloor \frac{|M^+|}{2} \rfloor$ such vertices (this is at most $2$). Finally for $T^-$ we have that even vertices in $J$ are endpoints of $B$ edges, therefore the only vertices with degree $1$ are in $[-x_b-1,-\frac{x_b-3}{2}-2]$. This is $\frac{x_b+5}{4}$ vertices.

This information is summarised in the table below.

\begin{table}[h!]
\centering
\begin{tabular}{c|l|l|l}
Set & Range in which endpoints lie & Endpoint type & Number of endpoints \\ \hline
$\{0\}$ & - & all & $1$ \\
$U^+$ & - & none & $0$ \\
$M^+$ & $[x_b+2,\frac{n-1}{2}-r]$ & all & $1$ to $4$ \\
$T^+$ & $[3,\frac{x_b-1}{2}]$ & odd & $\frac{x_b-3}{4}$ \\
$U^-$ & $[-\frac{n-5}{2},-\frac{n-1}{2}+r-1]$ & even & $\lceil \frac{n-17}{8} \rceil$ to $\lfloor \frac{n-15}{8} \rfloor$ \\
$M^-$ & $[-\frac{n-1}{2}-r,-x_b-3]$ & even & $0$ to $2$ \\
$T^-$ & $[-x_b-1,-\frac{x_b-3}{2}-2]$ & even & $\frac{x_b+5}{4}$ \\
\end{tabular}
  \caption{Endpoints in each segment of $L$}
  \label{Table:1}
\end{table}

In total we have $\frac{n-2x_b+1}{2}$ vertices in $L$ with degree $1$ when $\frac{n-1}{2}$ is even, and $\frac{n-2x_b-1}{2}$ otherwise. This means $L$ consists of
\begin{equation*}
    \frac{n-2x_b+(-1)^{\frac{n-1}{2}}}{4}
\end{equation*}
paths. We now look at the behaviour of each path.

\begin{claim}\label{paths T- to T-}
Any path in $L$ with an endpoint in $T^-$ must have its other endpoint in $T^- \cup \{0\}$.
\end{claim}
\begin{proof}
Recall that all endpoints in $T^-$ are in the set $[-x_b-1,-\frac{x_b-3}{2}-2]$, and that only even vertices are endpoints.
Set $a_1 = -x_b-1$ and $a_2 = -\frac{x_b-3}{2}-2$, so our interval of endpoints can be written as $[a_1,a_2]$. We will work in stages, checking the largest endpoint in our interval at each step. We will show that the path from this largest vertex terminates at the smallest endpoint in our interval (or at $0$). We then remove these endpoints from our interval and start again, checking the largest endpoint, continuing until we have found all the paths with endpoints in $T^-$. We will start with step $0$.

Step $0$, our interval is $[a_1,a_2]$ and we check the path at the largest vertex, $a_2 = -\frac{x_b-3}{2}-2$. We know there is an $M_0$ edge at $a_2$, so our path must go through the vertex $-a_2=\frac{x_b-3}{2}+2 = -\frac{x_b-3}{2}+2(t-1)$. From the construction of $B$, we see that there is a $B$ edge at this vertex which our path must follow next. This takes the path to the vertex $b(-a_2)=x_b+1$, where we must again follow an $M_0$ edge to $-(x_b+1)=a_1$. We know that $a_1$ is even and is the smallest vertex in our interval $[a_1,a_2]$, so the path terminates here.

Every remaining endpoint is an even vertex in $[-x_b+1,-\frac{x_b-3}{2}-4]$. Indeed, our original interval was $[-x_b-1,-\frac{x_b-3}{2}-2]$ and we found $-(x_b+1)$ and $-\frac{x_b+1}{2}$ to be endpoints of the first path. Since $-(x_b+1) +1$ and $-\frac{x_b+1}{2}-1$ are odd, they cannot be endpoints of a path in $L$. Hence our new interval is $[a_1 +2,a_2-2]$. This completes step $0$.

Step $k$, our interval is $[a_1+2k,a_2-2k]$, note that $k \leq \frac{x_b+1}{8}$. We check the largest vertex in our interval, $-v=a_2-2k$. We know $-v$ has an edge to $v$ via $M_0$, and since $v>\frac{x_b+3}{2}$ we have that $v\in J'$ has an incident $B$ edge. In particular this $B$ edge is $(-\frac{x_b-3}{2}+2(2k-1), \frac{x_b+3}{2}+2k-1)$ where $v=\frac{x_b+3}{2}+2k-1$.

We know that $b^{-1}(v)$ is even, if $b^{-1}(v)=0$ then the path must end here since there is only one edge at vertex $0$. Also note that when $b^{-1}(v)=0$ we have $k = \frac{x_b+1}{8}$ and hence our interval consists of a single vertex, therefore we must be on the final step. Assume then that $b^{-1}(v)\neq 0$. In this case we have that $b^{-1}(v) < 0$ since we are assuming $k < \frac{x_b+1}{8}$.

From $b^{-1}(v)$ we move along the path to $|b^{-1}(v)|$, since this is an even vertex and as $|b^{-1}(v)|<|-\frac{x_b-3}{2}|$, there must be another $B$ edge here. Since $b^{-1}(v)$ is even, there are exactly $|b^{-1}(v)|-1$ even vertices between $b^{-1}(v)$ to $|b^{-1}(v)|$ (not inclusive). This means there are $|b^{-1}(v)|-1$ $B$ edges between those at $b^{-1}(v)$ and $|b^{-1}(v)|$. In other words, the $B$ edge at $|b^{-1}(v)|$ must be $(-\frac{x_b-3}{2}+2(\ell-1), \frac{x_b+3}{2}+\ell-1)$ where $\frac{x_b+3}{2}+\ell-1=v+|b^{-1}(v)|$. Now, recall that $|b^{-1}(v)|=\frac{x_b-3}{2}-2(2k-1)$, and that $v=\frac{x_b-3}{2}+2+2k$. Using these we get $\frac{x_b+3}{2}+\ell-1=x_b+1-2k$. We then follow the $M_0$ edge at $x_b+1-2k$ to get to $-(x_b+1-2k)= a_1 +2k$.

Clearly then $a_1 +2k$ is even and inside the interval $[a_1+2k,a_2-2k]$. So it must be an endpoint in $L$, and therefore our path ends in $T^-$ as claimed. We remove vertices from our interval to give the new interval of $[a_1+2k+2,a_2-2k-2]$.

We can repeat the above process until we have seen that all paths with endpoints in $T^-$ have their other endpoint at $0$ or in $T^-$.
\end{proof}

Note that due to the structure shown in the above, it is easy to join these paths into a single path using short edges.

\begin{claim}\label{joining paths in T-}
We can use only $2$-type edges to connect the paths with endpoints in $T^-$ into a single path.
\end{claim}
\begin{proof}
Use the connecting edges
\begin{multline*}
    C_B= \Bigg\{\left(-\frac{x_b-3}{2}-2,-\frac{x_b-3}{2}-4\right), (-x_b+1,-x_b+3),  \dotsc,\\ \left(-\frac{x_b-3}{2}-2-4k, -\frac{x_b-3}{2}-4-4k\right), (-x_b-1+2+4k, -x_b-1+4+4k),\dotsc \Bigg\}.
\end{multline*}
Note that this connected path has one end vertex at $(-x_b-1)$, which is the vertex with the smallest label in $T^-$, and the other end is either $0$ or the vertex $-\frac{3x_b+7}{4} \in T^-$.
\end{proof}

By doing this, we have replaced $\lceil\frac{x_b+5}{8}\rceil$ paths with a single path.
 
Now we consider vertices in $T^+$. Firstly, for $n\geq 44$ we have that vertices $3$ and $5$ are endpoints of the same path, this path uses the all $R_1$ edges along with the first two edges of $R_2$. When $n < 44$ there is a $B$ edge at vertex $5$ meaning it cannot be an endpoint. This leaves us to check odd vertices in $[7,\frac{x_b-1}{2}]$.

\begin{claim}\label{paths T+ to U-}
Every odd vertex in $[7,\frac{x_b-1}{2}]$ is the beginning of a path in $L$ that terminates in $U^-$.
\end{claim}
\begin{proof}
Every vertex $v \in [7,\frac{x_b-1}{2}]$ has an incident $M_0$ edge to $-v \in T^-$, since $v$ is odd $-v$ has an incident $R_2$ edge. By the definition of the $R_2$ edges and of $U^+$, this edge must be $(-v,r(-v))$ where $r(-v) \in U^+$. Again from $r(-v)$ the path must follow the $M_0$ edge to $-r(-v) \in U^-$. Now we have two cases, either $r(-v)$ is even and hence there is no $R$ edge at $-r(-v)$ (also clearly no $B$ in $U^-$), or $r(-v)$ is odd and there is another $R$ edge that leads to a vertex in $U^+$.

For the case where $r(-v)$ is even, since there are no other edges at $-r(-v)$ the path must terminate here. This follows the claim as $-r(-v) \in U^-$. So we assume $r(-v)$ is odd and that $(\ell,r(\ell))$ is the $R$ edge with $-r(-v) = \ell$. By definition of $U^+$ we have that $r(\ell) \in U^+$, so the path must continue via $M_0$ to $-r(\ell) \in U^-$. Once again, if $r(\ell)$ is even the path terminates here and proves the claim. So we must assume $r(\ell)$ is odd, in which case we follow the path to another vertex in $U^+$ via $R$ as before, returning to $U^-$ via $M_0$. Here we are faced again with termination of the path at an even vertex, or continuing to $U^+$ and subsequently back to $U^-$ at odd vertex. Since the graph is finite this process must end, clearly the only location the process finishes is at an even vertex in $U^-$.
\end{proof}

Once again we can use this structure to join together the above paths into fewer longer paths. Indeed, using the edges $(5,7), (9,11), (13,15),\dotsc$ for all odd vertices in $[7,\frac{x_b-1}{2}]$ immediately halves the number of paths. But we can do better than this.

\begin{claim}\label{joining paths T+ to U-}
We can use only $2$-type edges to connect the paths with endpoints in $T^+$ into at most $\frac{1}{2}(\lceil \log_2(\frac{x_b-11}{4}) \rceil +1)$ paths.
\end{claim}
\begin{proof}
Recall that we are considering $n \geq 44$.

Consider the path $P_1$ from vertex $7$, and $P_2$ from vertex $9$. After three edges ($M_0$, $R$, and $M_0$), $P_1$ must go through vertex $-\frac{n-5}{2}$, and $P_2$ through $-\frac{n-7}{2}$. Clearly, one of $-\frac{n-5}{2}$ and $-\frac{n-7}{2}$ is even and therefore is the endpoint of the path. Let $v=7$ if $\frac{n-5}{2}$ is even, and $v=9$ otherwise, and let $-u$ be the other endpoint of the path at $v$. Then we must have that the path from $v+4$ terminates at $-u+2$, and in general the path at $v+4k \in [7,\frac{x_b-1}{2}]$ terminates at $-u+2k$ for $k \in \mathbb{N}$.

Similarly, consider the paths from vertices $v+2$ and $v+6$, $P_3$ and $P_4$. After $3$ edges, both paths are at an odd vertex in $U^-$, in particular, $P_3$ is on $-u+1$ and $P_4$ is on $-u+3$. This means they are on odd vertices in $U^-$ and hence the next edge in both paths is an $R$ edge to $U^+$. Let $w=r(-u+1)$, then we must have that $(-u+3,w+1)$ is the $R$ edge at $-u+3$. The paths $P_3$ and $P_4$ must then follow $M_0$ edges to $-w$ and $-w-1$ respectively. Once again, one of $-w$ and $-w-1$ is even and hence the path terminates here, the other must continue along another $R$ edge. Let $w'$ be the even vertex out of $w$ and $w+1$, and let $v'$ be the vertex in $T^+$ on the path to $w'$ (either $v+2$ or $v+6$). Then we have that the path from $v'+8$ must terminate at $-w'-2$, and similarly the path from vertex $v'+8k \in [7,\frac{x_b-1}{2}]$ terminates at vertex $-w'-2k$.

In the first stage we found that paths from every other vertex in $[7,\frac{x_b-1}{2}]$ terminate at consecutive endpoints in $U^-$. Let $T_1 \subseteq [7,\frac{x_b-1}{2}]$ be those vertices. Then we looked at $[7,\frac{x_b-1}{2}] \setminus T_1$ and found that half of these vertices lead a path that terminates after $5$ edges and have consecutive endpoints in $U^-$. Let $T_2 \subseteq [7,\frac{x_b-1}{2}] \setminus T_1$ be those vertices found to be in paths of length $5$. Then we can continue and consider odd vertices in $[7,\frac{x_b-1}{2}] \setminus (T_1 \cup T_2)$. By using the same argument as above we find again that half of these vertices are endpoints to paths that have $7$ edges and with the $U^-$ endpoints all consecutive.

By the end of this process we have at most $\lceil \log_2(\frac{x_b-11}{4}) \rceil +1$ sets $T_i$. For each set $T_i$, let $U_i$ be the set of vertices in $U^-$ which are endpoints of paths starting in $T_i$.

We are now ready to define our connecting edges. First we use the edges
\begin{equation*}
    C_0=\{(9,11), (13,15),\dotsc \}
\end{equation*}
for all odd vertices in $[9,\frac{x_b-1}{2}]$. Since every path starting in $T^+$ terminates in $U^-$ adding these edges does not create any cycles. In particular, these edges connect a vertex in $T_1$ with a vertex in $T_i$ where $i \neq 1$. This means that the new path has one endpoint in $U_1$ and the other in $U_i$.
After adding these edges, the vertex $7$ remains an endpoint in $T^+$, but the only other possible endpoint is the vertex $\frac{x_b+3}{2}-2=\frac{x_b-1}{2}$.

Let $T_i = \{ v^i_1,v^i_2,\dotsc,v^i_{\ell_i} \}$ where $v^i_j < v^i_k$ whenever $j<k$. Similarly, we let $U_i = \{ u^i_1,u^i_2,\dotsc,u^i_{\ell_i} \}$ where the path from $v^i_j$ ends at $u^i_j$. The next connecting edges we add will be
\begin{equation*}
    C_1=\{(u^1_1,u^1_2), (u^1_3,u^1_4),\dotsc\}.
\end{equation*}
Again, this cannot create a cycle since no path begins and ends in $U_1$. Observe that the edge $(u^1_{2i-1}, u^1_{2i})$ joins a path through vertex $v^1_{2i-1}$ with a path through $v^1_{2i}$. If $v^1_{2i-1}=k$ then $v^1_{2i}=k+4$, meaning that $k$ is joined to $k' \in \{k-2,k+2\}$ by an edge in $C_0$ (or $k=7$ has degree $1$) and $k+4$ is joined to $k'+4$ by a $C_0$ edge (unless $k+4=\frac{x_b-1}{2}$ has degree $1$). Since $k'$ and $k'+4$ differ by $4$ and neither are in $T_1$ we must have that either $k' \in T_2$ or $k'+4 \in T_2$ but cannot have both. This means that the new path containing the edge $(u^1_{2i-1}, u^1_{2i})$ cannot have both endpoints in $U_2$.

Again, we now do an analogous process for connectors in $T_2$. That is, we use connectors $C_2=\{(u^2_1,u^2_2), (u^2_3,u^2_4),\dotsc\}$, noting that we have not created a cycle since no path has two endpoints in $U_2$. Using an analogous argument to the above we also have that no new path has two endpoints in $U_3$. We continue this way until we have defined connector sets $C_0,C_1,\dotsc,C_{\ell}$ where $\ell \leq \lceil \log_2(\frac{x_b-11}{4}) \rceil +1$.

Finally we add the special connecting edge $(5,7)$.

It is left to count the number of paths we now have, as well as the number of paths we replaced. Recall that there were $\frac{x_b-11}{4}$ endpoints in $[7,\frac{x_b-1}{2}]$ and hence $\frac{x_b-11}{4}$ paths from $T^+$ to $U^-$, along with the path from $3$ to $5$. We count the new number of paths by again counting the number of endpoints. Note that the vertex $3$ is an endpoint and $\frac{x_b-1}{2}$ may also be an endpoint (if it is $1 \mod 4$). There are now no other endpoints in $T^+$. Since we connected vertices in $U_i$ by pairing them, there can be at most $1$ endpoint in each $U_i$, and hence at most $ \lceil \log_2(\frac{x_b-11}{4}) \rceil +1$ endpoints in $\cup_{1 \leq i \leq \ell} U_i$. Therefore, in total we have replaced $\frac{x_b-7}{4}$ paths with at most
\begin{equation*}
     \frac{ \lceil \log_2(\frac{x_b-11}{4}) \rceil +1}{2}
\end{equation*}
new paths. It is left only to note that we have only used $2$-type edges in this construction as we have connected consecutive odd vertices in $T^+$ and consecutive even vertices in each $U_i$.
\end{proof}

At this point, the total number of paths in $L \cup C_B \cup C_0 \cup \dots \cup C_{\ell} \cup \{(5,7)\}$ is given by
\begin{equation*}
    \frac{n-2x_b+(-1)^{\frac{n-1}{2}}}{4} - \left\lceil\frac{x_b+5}{8}\right\rceil +1 - \frac{x_b-7}{4} +  \frac{\left \lceil \log_2(\frac{x_b-11}{4}) \right\rceil +1}{2}.
\end{equation*}
This is at most
\begin{equation*}
    \frac{n+16\log_2n+163}{32}.
\end{equation*}

Let $L'=L \cup C_B \cup C_0 \cup \dots \cup C_{\ell} \cup \{(5,7)\}$ be the linear forest at this point. Finally we will connect $L'$ using any suitable edges, call the set of these additional edges $C_A$. Let $P$ be the path $P = L \cup C$, where $C = C_B \cup C_0 \cup \dots \cup C_{\ell} \cup \{(5,7)\} \cup C_A$. 

Note that $|C_A| \leq \frac{1}{32}(n+16\log_2n+131)$ by the number of paths remaining in $L'$. Furthermore, we are able to choose the edges of $C$ in such a way that for any $x$, if there are more than two $x$-type edges in $P$ then they are not spaced out evenly on $K_n$. This ensures that all edges of the same type are separated from each other by the rotations of $P$. All the edge types in $C_A$ will end up in $D$, so apart from ensuring the edges are not equally spaced there are no other conditions for these edges to follow.

\begin{claim}\label{no even spaced in P}
We can join the paths in the linear forest $L'$ into a single path $P$ such that there are no evenly spaced $x$-type edges in $P$.
\end{claim}

\begin{proof}
Consider the linear forest $L'$ and its endpoints (see summary in Table \ref{Table:2}).

\begin{table}[h!t]
\centering
\resizebox{\linewidth}{!}{
\begin{tabular}{c|l|l|l}
Set & Endpoints in set & Other endpoint of path & Reference \\ \hline
$\{0\}$ & all & $-x_b-1$ or $U^-$ & Claim \ref{joining paths in T-}\\
$U^+$ & none & - & See Table \ref{Table:1}\\
$M^+$ & all & $U^-$ and possibly $M^-$ & See $M^-$\\
$T^+$ & $3$ and possibly $\frac{x_b-1}{2}$ & $U^-$ or the pair $\{3,\frac{x_b-1}{2}\}$ & Claim \ref{joining paths T+ to U-}\\
$U^-$ & some evens & $U^-$, $M^+$, and possibly $T^+$ and $0$ & See $M^+, T^+, \{0\}$\\
$M^-$ & all evens & $M^+$ & Length $1$ paths\\
$T^-$ & $-x_b-1$ and possibly $-\frac{3x_b+7}{4}$ & $0$ or the pair $\{-x_b-1,-\frac{3x_b+7}{4}\}$ & Claim \ref{joining paths in T-}\\
\end{tabular}}
  \caption{Path connections between each segment of $L'$}
  \label{Table:2}
\end{table}

Starting with $C_A$ empty, we add edges as follows. First add the edge $(0,3)$, note that this is a $3$-type edge. Since $R$ and $B$ only contain edges of large type, there is only one $3$-type edge in $L'$. In particular, since $n$ is odd, $L' \cup \{(0,3)\}$ does not have evenly spaced $3$-type edges.

Next, suppose the vertex $v=-\frac{3x_b+7}{4}$ is indeed and endpoint. This means that the distance from $v$ to some endpoint in $U^-$ is at most $\frac{3n-29}{16}<\frac{n-1}{4}$. Indeed, in $L$ there was an endpoint at every even vertex of $[-\frac{n-5}{2},-\frac{n-1}{2}+r-1]$ until we added $C_0 \cup \dots \cup C_{\ell} \cup \{(5,7)\}$. This left at least $\frac{n+3}{16}$ vertices in $U^-$. Further note that any edge from $v$ to an endpoint in $U^-$ must have even edge type since both endpoints are on even vertices. Let $e_v$ be any such edge, note that $L' \cup \{(0,3), e_v\}$ is still a linear forest.

We now move on to the vertices in $M^+$. If $|M^+|=1$ leave this vertex as an endpoint. If $|M^+|=2$ then use a $1$-type edge $e_+$ to join them, noting that this creates a path with one vertex in $M^-$ and the other in $U^-$. This cannot create a cycle since we have not added any edges to vertices in $M^-$. If $|M^+|=3$, join with $e_+$ as in the previous case, and leave one as an endpoint. Finally, if $|M^+|=4$ join with two $1$-type edges $e_+$ and $e_+'$, again this does not create a cycle. Denote by $E_M$ the set of edges added within $M^+$, so $E_M$ depends on $|M^+|$ and is equal to one of $\emptyset$, $\{e_+\}$ or $\{e_+,e_+'\}$. Note that $L' \cup E_M$ does not contain evenly spaced $1$-type edges. Indeed, when $E_M = \emptyset,\{e_+\}$ then $L' \cup E_M$ contains at most $2$ type $1$ edges, and therefore they cannot be evenly spaced. When $E_M =\{e_+,e_+'\}$ there are exactly three $1$-type edges in $L' \cup E_M$ and two of them, $e_+$ and $e_+'$, have clockwise distance $2$ between them. Therefore there are no evenly spaced $1$-type edges.

At this point $C_A = \{(0,3), e_v\} \cup E_M$.

Now, $L' \cup C_A$ is a linear forest and has at most two endpoints outside of $U^- \cup M^- \cup \{-x_b-1\}$, the vertices $\frac{x_b-1}{2}$, and some vertex $u \in M^+$. Since $u$ and $\frac{x_b-1}{2}$ are not endpoints of the same path (see construction in \ref{joining paths T+ to U-}), we simply need to join all endpoints within $U^- \cup M^- \cup \{-x_b-1\}$ to create our path $P$. Note that since all endpoints in this region are even vertices, this will only require edges of even edge type. Moreover it only requires edges of type $\frac{n-1}{4}$ or shorter. Indeed, the longest possible edge in this interval is $(-x_b-1,-\frac{n-5}{2})$, which has type at most $\frac{n+1}{4}$, but since there are at least $\frac{n+3}{16}-1$ endpoints remaining in $U^-$ (as we saw earlier in this proof) we do not need to use this longest edge. All other edges in the interval have type at most $\frac{n-1}{4}$ as required.

We add as many of these short even edges to $C_A$ as required, until $L' \cup C_A$ is a path. Then we set $P=L' \cup C_A$.

It is left to check that $P$ has no evenly spaced edge types. Note that $L$ clearly does not since $L$ contains at most $2$ edges of each type. So we only need to check the edge types appearing in $C$. Note that among all the edges in $C$, the only edges that are not even are the edges in $E_M$ and the edge $(0,3)$. The remaining edges in $C$ are all even and have edge type at most $\frac{n-1}{4}$. We have seen that the $3$-type edge and the $1$-type edges do not create a problem, it is left to check the even edges of $C$.

Consider edges of even type in the linear forest $L$, with starting vertex in $[\frac{x_b+3}{2},\frac{n-1}{2}]\cup[-\frac{n-1}{2},-\frac{n-3}{2}]$. Such edges must all be from $R$. Indeed, even-type $M_0$ edges all start on a negative vertex, and in particular the $M_0$ edges at $-\frac{n-1}{2}$ and $-\frac{n-3}{2}$ both have odd type. Also all $B$ edges have their starting vertex in the interval $J$, which does not intersect $[\frac{x_b+3}{2},\frac{n-1}{2}]\cup[-\frac{n-1}{2},-\frac{n-3}{2}]$. We also know that all edges in $R$ have a type from $[\lceil \frac{n-1}{4} \rceil,\frac{n-1}{2}]$.

The only edges added to this interval to create $P$ are the edges in $E_M$. We know that $E_M$ only contains edges with odd type. In other words, in $P$ no edges of even type from $[2,\frac{n-1}{4}]$ have a starting vertex in the interval $[\frac{x_b+3}{2},\frac{n-1}{2}]\cup[-\frac{n-1}{2},-\frac{n-3}{2}]$. This means that if $e_1,\dotsc,e_m$ are all the $x$-types in $P$ with even $x\in [2,\frac{n-1}{4}]$, and $cd(e_i,e_{i+1})=\frac{n}{m}$ then $\frac{n}{m}\leq \frac{3n+17}{8}$. This forces $m < 3$, since $n$ is odd we cannot have any evenly spaced $x$-type edges when $m=2$.

Therefore there are no evenly spaced $x$-type edges in $P$.
\end{proof}

We use $P$ as the base path for Theorem \ref{rotations plus fixings} to give a separating path system for $K_n$. The size of this family is dependent on the size of $D =[\frac{n}{2}] \setminus F$.

\begin{claim}\label{size of D}
Let $P$ be the path defined in this section. Then $P$ is an $F$-separator path, where $D$ satisfies the following.
\begin{equation*}
    |D\cup\{1\}| \leq \frac{5n+16\log_2n+167}{32}
\end{equation*}
\end{claim}

\begin{proof}
From \ref{MRB sat 13} we know that in $L$ all edge types in $R\cup B$ appear exactly twice, and the remainder appear exactly once. This means there are at most
\begin{equation*}
    \frac{n-1}{2}-\frac{n-9}{4}-2-\frac{x_b+1}{2} \leq \frac{n+9}{8}
\end{equation*}
edge types in $L$ that appear exactly once. We put all of these edge types into $D$. Note that, by the construction of $R$ and $B$, these are the shortest $\approx \frac{n}{8}$ edge types. In particular $1,2 \in D$.

Moreover, we know by \ref{MRB sat 13} that $L$ is an $F_L$-separator where $F_L = \{x: \text{there is an } e\in R\cup B \text{ with edge type } x\}$. To turn $L$ into $P$ we added various connecting edges $C$. The only way an edge type can be in $F_L$ but not in $F$ is if some edge in $C$ has edge type from $F_L$. Therefore, for the path $P$, the types associated with these $C$ edges must also appear in $D$, and the remaining types in $F$.

All edges in $ C_B \cup C_0 \cup \dots \cup C_{\ell} \cup \{(5,7)\}$ have type $2$, which are already accounted for. We used a final $\frac{1}{32}(n+16\log_2n+131)$ edges in $C_A$. This means $D$ contains at most
\begin{equation*}
    \frac{n+9}{8} + \frac{n+16\log_2n+131}{32}= \frac{5n+16\log_2n+167}{32}
\end{equation*}
edge types.
\end{proof}

It is natural to try adapting these methods to the strong version of the problem. We note that the way to adapt the path system given in Theorem \ref{upper bound} using methods from this paper would be to use more fixing paths (as in Lemma \ref{Q paths}). However, this gives an upper bound of approximately $2n$ which is not an improvement on the previously known result Theorem \ref{B et al upper bound}, although it is constructive. We omit the details. It is an interesting open problem to construct a strongly separating path system for $K_n$ with fewer than $2n$ paths.

\section{Acknowledgements}
The author was supported by an EPSRC doctoral studentship.

\end{document}